\documentclass[10pt,a4paper, envcountsame]{amsart}
\usepackage[usenames,dvipsnames]{color}

 \usepackage[all]{xy}
\usepackage{amssymb}

\usepackage{tikz}
\usetikzlibrary{arrows,snakes,positioning,backgrounds,shadows}

\usepackage{enumerate}
\usepackage{enumitem}
\usepackage{multicol}
\usepackage{amssymb,amsmath}
\usepackage{amsthm}
\usepackage{turnstile}
\usepackage{textcomp}
\usepackage{gensymb}
\usepackage{graphicx}
\usepackage{subcaption}
\usepackage{mathtools}
\usepackage{hhline}
\usepackage{array,multirow}
\usepackage{float}
\usepackage{framed}

\theoremstyle{plain}
\newtheorem{theorem}{Theorem}[section]
\newtheorem{lemma}[theorem]{Lemma}
\newtheorem{proposition}[theorem]{Proposition}
\newtheorem{corollary}[theorem]{Corollary}
\newtheorem{fact}[theorem]{Fact}

\newtheorem{remark}[theorem]{Remark}

\newtheorem{definition}[theorem]{Definition}

\DeclareMathOperator{\U}{\mathcal{U}}
\DeclareMathOperator{\V}{\mathcal{V}}

\title{Computable topological groups}

\author{}
\date{\today}

\begin{document}

\author{Heer Tern Koh}
\email{heertern001@e.ntu.edu.sg}

\author{Alexander G. Melnikov}
\address{Victoria University of Wellington, Wellington, New Zealand, \and Sobolev Institute of Mathematics, Novosibirsk, Russia}
\email{alexander.g.melnikov@gmail.com}

\author{Keng Meng Ng}
\address{Nanyang Technological University, Singapore}
\email{kmng@ntu.edu.sg}

\thanks{\rm Melnikov was supported by the Mathematical Center in Akademgorodok under agreement No.~075-15-2019-1613 with the Ministry of Science and Higher Education of the Russian Federation.  Ng was  supported by the Ministry of Education, Singapore, under its Academic Research Fund Tier 1 (RG23/19). This work was also partially supported by Rutherford Discovery Fellowship (Wellington) RDF-MAU1905, Royal Society Te Aparangi.}

\maketitle

\vspace{-10mm}
\begin{abstract} We investigate what it means for a (Hausdorff, second-countable) topological group to be computable. We compare several potential definitions in the literature.
We relate these notions with the well-established definitions of effective presentability for discrete and profinite groups, and compare these results with similar results in computable topology.
Most of these definitions can be separated by counter-examples.
Remarkably, we prove that two such definitions are  equivalent for locally compact Polish and abelian Polish groups.
More specifically, we prove that in these broad classes of groups,  every computable topological group admits a right-c.e.~(upper semi-computable)
presentation with a left-invariant metric, and a computable dense sequence of points.
 In the locally compact case, we also show that if the group is additionally effectively locally compact, then we can produce an effectively proper left-invariant metric.

\end{abstract}
\tableofcontents

\section{Intorduction}
\subsection{A brief overview}
Our  paper contributes to a fast developing branch of effective mathematics which combines methods of computable algebra~\cite{ErGon,AshKn, Ershov} with tools of computable analysis~\cite{Brattka.Hertling.ea:08,Wei00,PourElRich}
to study computable presentations of topological groups.   Maltsev~\cite{Ma61},  Rabin~\cite{abR},  Higman \cite{Hig},  Metakides and Nerode
\cite{MetNer79}  and  others (e.g., \cite{sinf, MeMo,lupini,ArnoHaar, tdlc}) suggested various notions of computability  for various classes of groups.
Whenever a theory emerges, one of the first tasks is to establish  that its key definitions are robust by supporting them with enough non-trivial examples and deep results.
Another important task is to compare the most important definitions and see if they are equivalent. In the context of algorithmic group theory, one of the most well-known examples of such a separation result is the celebrated work of Novikov~\cite{Nov:55} and Boone~\cite{Boone:59} who proved that not every finitely presented group has decidable Word Problem.
Such investigations often lead to a deeper understanding of the notions of computable presentability that are being studied. For instance, in his search for a more elegant proof of the Novikov-Boone theorem, Higman~\cite{Hig} discovered that `recursively presented' groups are exactly the subgroups of finitely presented groups, thus characterising one of two notions of presentability for groups in terms of  the other.

In the present paper we prove several \emph{positive} characterization-type results that we believe are fundamental to the emerging theory of computable topological groups. As the main result of the paper, we prove that in the locally compact and abelian case, a seemingly weak notion of effective topological presentability is indeed equivalent to the seemingly much stronger notion of right-c.e.~Polish presentability. The second main result of the paper further improves the result, but under one extra assumption of effective local compactness. (All these terms will be defined in due course.) Finally, we also support these positive results with counter-examples that separate several notions of computable presentability up to topological group isomorphism.  These counter-examples and their proofs relate our notions with the aforementioned `recursive' groups studied by Higman, computable groups  as defined by Maltsev~\cite{Ma61} and Rabin~\cite{abR}, and with `recursive' profinite groups investigated by Metakides and Nerode
\cite{MetNer79} and Smith~\cite{Smith1}. Indeed, we will see that in many cases some of these definitions turn out to be equivalent.

\

Before we formally state our results, we give a bit more background and briefly discuss the related literature.

\subsection{Notions of computable presentability} 
The most well-established notions of computable presentability of groups are restricted to discrete and profinite groups, as we discuss below.

We have already mentioned the notion of a `recursive presentation'~\cite{Hig} which is standard in combinatorial group theory. We call such presentations \emph{computably enumerable} (c.e.), or $\Sigma^0_1$, since the term ``recursive presentation'' can mean many different and non-equivalent notions (as we shall see shortly).
These presentations are groups of the form $F_\omega/ H$, where $F_\omega$ is the standard  reduced-word presentation of the free group upon $\omega$ generators and $H$ is its computably enumerable normal subgroup. The equality relation modulo $H$, the ``Word Problem'', does not have to be computable even in a finitely presented group~\cite{Nov:55, Boone:59}.
Perhaps motivated by this early fundamental result,
 Mal'cev~\cite{Ma61} and Rabin~\cite{abR} suggested a stronger notion of computable presentability for a group. In the notation above, a group is \emph{computably presented} if it is isomorphic to
$F_\omega/ H$, where $H$ is a computable subset of $F_\omega$.  In other words, a group is computable if it is c.e.~and the Word Problem is decidable in the group. (We note that the terms ``computable'' and ``recursive'' were often used to mean the same thing in the early literature, and this can potentially lead to much confusion in the context of group presentations.) Equivalently, a countably infinite group is computable if its domain is $\mathbb{N}$ and the group operations are represented by computable functions.
 This notion can be extended to an arbitrary discrete algebraic structure in the obvious way. The notion is well-established and is a central notion in the technically deep theory of computable algebraic structures, see \cite{AshKn,ErGon}.

From the perspective of computable presentability, the second most well-understood class is the class of profinite groups.  Metakides and Nerode
\cite{MetNer79} and then La Roche~ \cite{LaRo1,LaRothesis} and Smith~\cite{Smith1,SmithThesis} studied   ``recursively presented''  and ``co-r.e.-presented'' profinite groups.
To define computability in this class  we just need to say that the inverse system representing the group is (in some sense) algorithmic; we omit the definitions.
Among other results, they prove the effective versions of Galois correspondence that relate recursive and co-r.e.~presentations with
 computable and computably enumerable field extensions, respectively. Another duality established in the cited papers is a Stone-type duality between recursive groups and decidable classes in $2^{\omega}$, and between co-r.e.~presented groups and $\Pi^0_1$ classes in $2^{\omega}$. In both cases, the classes are also equipped with group operations.
 The much more recent paper~\cite{Pontr} proves  that  the recursive profinite abelian groups are exactly the Pontryagin duals of computable discrete torsion abelian groups.
There are many ways to apply these dualities to establish that co-r.e.~presentability is strictly weaker than recursive presentability for profinite groups.

One of the characteristic features of the results briefly discussed above is that \emph{the notions of computability in the profinite and the discrete case are interconnected and related via dualities of various kinds}. Are these two subjects just pieces of a bigger puzzle? More specifically, can we develop a general theory of computable topological groups that is not restricted to the profinite and discrete case?

\

The situation becomes much more complex when we consider groups which are neither profinite nor discrete. Following the analogy with the discrete and profinite cases,
a computable topological group should mean a \emph{computable space} together with \emph{computable group operations}.
In the present paper, we are mainly interested in Polish groups but some of the technical results proven here also work for Hausdorff second-countable groups.
Thus, for the remaining of the paper we will adopt the convention:
\begin{center}
\emph{All our groups and spaces are Hausdorff and second countable.}
\end{center}
Indeed, we are mainly interested in locally compact groups, and it well-known that every Hausdorff and second countable locally compact space is Polish.
Unfortunately, even in the nice case of a Polish(able) space, it is not even clear what ``computable space'' should mean exactly.

 Computable topology is notorious for its zoo of different notions of computability for a topological space (and a topological group).
In contrast with effective algebra~\cite{ErGon,AshKn} where all standard notions of computable presentability in common classes had been separated more than half a century ago (e.g., Novikov~\cite{Nov:55}, Boone~\cite{Boone:59},  Feiner~\cite{feiner1970}, Khisamiev~\cite{Hisa2}, Odintsov and Selivanov~\cite{OdSel-89}),
some  of the basic notions of computable presentability in topology have been separated only very recently
~\cite{topsel,uptohom,lupini,bastone,separating}. 
Arranged from strong to weak, some common notions of effective presentations for a Polish space are as follows:

\begin{multline*}
\text{computably compact } \longrightarrow \text{ computable Polish}\\  \longrightarrow \text{ right-c.e~Polish } \longrightarrow \text{ computable topological}
\end{multline*}

\noindent All these notions will be formally defined in the preliminaries and also briefly discussed below; we only mention here that the first one clearly works only for compact spaces, and that all implications are known to be \emph{strict up to homeomorphism}, as established in~\cite{topsel,bastone,separating,lupini,EffedSurvey}.
This means that, in each case, there is a space that is effectively presentable in the weaker sense but is not homeomorphic to any space effectively presentable in the stronger sense\footnote{For the last implication see~\cite{separating}, for right-c.e.~vs.~computable Polish we cite \cite{bastone}, and for the first implication see~\cite{topsel,lupini,EffedSurvey}. Also, for the closely related notion of a left-c.e.~Polish space which will not be used in this paper, we cite~\cite{separating}.}.

Each of these four notions leads to a definition of a computable group.
In each case, a `computable group' would mean  `computable space' + `computable group operations', where computability of operations is understood in terms of approximations, i.e., effective operators (to be clarified).

Up to topological group isomorphism, are  these four  notions non-equivalent? How are these notions related to the well-established approach in the discrete case? What about the profinite case? We will answer these questions shortly. 
But first, we 
clarify and compare the notions in the diagram for \emph{spaces}.

\

There are several variations of the definition of a computable topological space
that can be found in, e.g.,  Kalantari and Weitkamp \cite{Kal1} and Spreen \cite{Spreen3}.
We will use the following, perhaps the \emph{weakest possible}, approach. A \emph{computable topological presentation} is given by a countable base of topology $(B_i)_{i \in \omega}$ consisting of non-empty basic open sets together with the c.e.~set $W$ that allows one to list intersections in the following weak sense: $$B_i \cap B_j = \bigcup \{ B_k\,\colon (i,j,k) \in W\}. $$
The standard examples of computable Polish spaces include right-c.e.~Polish spaces that will be defined shortly.
However, note that, in general, the definition of a computable topological space is \emph{point-free}. This feature can be easily exploited to show that there is a Polish space that is computable topological but is not homeomorphic to any right-c.e.~Polish space~\cite{separating}. Indeed, it follows from the simple proof in the companion paper \cite{separating} that, in general, \emph{for a computable topological  locally compact (Polish) space there is no bound on the complexity of its  Polish presentation, up to homeomorphism.}
Nonetheless, the notion of a computable topological space is quite popular in the literature, but it usually comes with some extra additional assumptions on top of the base weak definition.

\

The classical notion of a \emph{computable Polish space}   can be traced back to Ceitin~\cite{metric1} and Moschovakis \cite{metric2}. We say that a Polish space is computable Polish if there is a countable dense subset $(x_i)_{i \in \omega}$ and a complete metric $d$ compatible with the topology such that $d(x_i, x_j)$ can be uniformly computed to precision $2^{-n}.$ If we only require that the real $d(x_i, x_j)$ can be effectively approximated from above by enumerating its upper cut (in some, perhaps unnatural, order), then we get the notion of a right-c.e.~Polish space, also known under the name upper-semicomputable Polish space.
The reason why the right-c.e.~case is so important in the literature is because it is \emph{the} standard example of a computable topological space. Also, it is known that Stone duality associates `effectively compact' right-c.e.~Stone spaces with c.e.~presented Boolean algebras~\cite{bastone}. In particular, it follows from results in \cite{topsel,uptohom} and the aforementioned classical result of Feiner \cite{feiner1970} that there is a right-c.e.~Polish space not homeomorphic to a computable Polish one; see \cite{bastone} for a detailed explanation.

\

Finally, we  say that a space is computably (or effectively) compact if it is computable Polish and additionally,  we can effectively list all finite basic open covers of the space.
  It has been proven in  \cite{lupini,EffedSurvey,topsel} that there are compact spaces that are  computable   Polish but not homeomorphic to any
computably compact space. Interestingly, the proof in
\cite{lupini} builds a connected compact \emph{group} with this property. It follows from the proof in \cite{lupini} that there is a connected compact abelian \emph{group}
that has a computable Polish presentation (as a group, i.e., in which the operations are also computable), but so that its \emph{space}
is not homeomorphic to any computably compact space. Thus, at least one implication is known to be strict for groups. The notion of a computable compactness is clearly restricted to compact spaces (and groups) and  it won't be too important to us. However, its natural generalisation to  locally compact spaces
  will be useful in the present paper.






\

\subsection{Results}\label{sec:results} We are ready to discuss our results. Recall that in the companion paper \cite{separating} we illustrate that there is a computable topological locally compact Polish \emph{space} that is not homeomorphic to any arithmetic (or even analytic, and beyond) Polish space, late alone a right-c.e.~Polish space. Recall also that in a (right-c.e.~or computable) Polish presentation, we demand the existence of a dense computable sequence and that the metric is complete. However, the definition of a computable topological group is \emph{point-free}.
The principal result of the present paper is the following:

\begin{theorem}\label{main:theo}
For a Polish group $G$ that is either abelian or locally compact, the following are equivalent:
\begin{enumerate}
\item $G$ has a computable topological presentation,
\item $G$ has a right-c.e.~Polish presentation.
\end{enumerate}
Furthermore, in $(2)$ the metric can be taken left-invariant. (Or right-invariant.)
\end{theorem}

The implication $(2)\rightarrow (1)$ is obvious, but  $(1)\rightarrow (2)$ is both non-trivial and (we believe) unexpected.  
It should not be surprising that one of the crucial steps in the proof
is a new effective version of the classical Birkhoff-Kakutani metrization theorem;  this is  Theorem~\ref{thm:effbk}. The proof of the effective version  of Birkhoff-Kakutani  theorem requires much care since we use the rather weak  point-free approach to computable topological spaces. Even more care is needed to reconstruct the dense sequence from the point-free effective topology; this is Theorem~\ref{thm:completemet}.  We shall also explain why in the locally compact and in the abelian cases  the metric produced in Theorem~\ref{thm:effbk} is complete; this is not obvious at all, but several classical results from topological group theory will come to our aid.
The proof of  Theorem~\ref{main:theo} is spread through the paper; for the abelian case see Corollary~\ref{cor:abelian}, and for the locally compact case see Corollary~\ref{cor:locallycompact}.

 In Corollary~\ref{cor:ce} we also show that Theorem~\ref{main:theo} is \emph{sharp}
in the sense that, in general, we cannot produce a computable metric. Our counter-example is a discrete abelian group that admits a right-c.e.Polish copy but is not topologically isomorphic to any computable Polish group.  We also mentioned earlier that computable compact does not imply computable Polish among compact abelian groups. Combined with the Theorem~\ref{main:theo} and the aforementioned result in \cite{lupini}, we obtain that for abelian and for locally compact Polish groups and up to topological group isomorphism, the diagram looks like:

\

\begin{center}
computable topological

$\downarrow$ $\uparrow$

 right-c.e~Polish

 $\downarrow$ $\nuparrow$

 computable Polish

 $\downarrow$ $\nuparrow$

  computably compact

\end{center}

\noindent Clearly, for the latter implication the counter-example is compact; such examples can be found among connected compact~\cite{lupini} (as discussed above) and also among profinite~\cite{EffedSurvey} abelian groups.

\begin{remark}\rm
We mention another notion of computable presentability motivated by research in computable structure theory that we did not include into the diagram. 
 Classically, closed subgroups of $S_\infty$ are exactly the automorphism groups of discrete structures; see \cite{GaoBook}.
Every automorphism group of a discrete \emph{computable} structure is  a $\Pi^0_1$ (effectively closed) subgroup of a certain natural effective presentation of $S_\infty$; see \cite{sinf,tdlc}.
However, the converse fails~\cite{sinf,uptohom}.  For instance, it is known that a \emph{compact} (thus, profinite) $\Pi^0_1$ subgroup of $S_\infty$
 does not have to be topologically isomorphic to a $\Delta^0_\alpha$-Polish group for any fixed computable $\alpha$~\cite{uptohom}. Therefore, already for compact groups,  this notion of computable presentability is (much) weaker than the weakest definition of a computable topological group that we study in this paper. Strictly speaking, such presentations are not really computable since one has essentially no access to the evasive domain of the group.
\end{remark}

We now discuss the second  main result of the present article.
One of the two cases in Theorem~\ref{main:theo} is when the group is locally compact.
Struble \cite{Struble} showed that  a (Hausdorff,  second countable) locally compact topological group  admits a compatible left-invariant \emph{proper} metric; recall that a metric is proper if every closed bounded set is compact. Can we improve Theorem~\ref{main:theo} and show that, in a locally compact case, we can in fact produce a right-c.e.~proper metric?

In the literature, most \emph{effective} arguments that involve the use of compactness assume that the space satisfies a version
 of \emph{effective} compactness of some sort. One such notion we have already mentioned above. Remarkably, many definitions of effective compactness in the literature turn out to be equivalent; see~\cite{IlKi,EffedSurvey,Pauly}.
 We will formally define and discuss the notion of effective compactness that we chose in the preliminaries section. Roughly speaking, a set is effectively compact if we can list all of its covers by basic open balls.

 To have access to \emph{local} compactness, we need to generalize this notion to locally compact spaces. There are several definitions in the literature~\cite{Pauly,xu_grubba,Weih_Zheng}. We shall not attempt to compare these definitions up to homeomorphism. However, we suspect that they are perhaps non-equivalent.
 The notion suggested in \cite{Pauly} seems most suitable for our purpose. Roughly speaking, it it is a direct effectivisation of  the classical notion of local compactness that says that every point is contained in a compact neighbourhood. We will define it formally in the preliminaries.

 We are ready to state the second main result of the paper:

 \begin{theorem}\label{main2:theo}
Every locally compact computable topological group admits a  proper right-c.e.~Polish presentation in which the metric is left-invariant and proper. \end{theorem}
Furthermore, the metric produced in the theorem is itself effectively locally compact and indeed  \emph{effectively proper} in the sense that will be clarified in the preliminaries.
We also point out that, in both Theorem~\ref{main:theo} and Theorem~\ref{main2:theo}, the right-c.e.~Polish presentations that we build are \emph{computably homeomorphic} (indeed, \emph{effectively compatible}) to the given computable topological presentation of the group. These standard notions will be clarified later.
Theorem~\ref{main2:theo} will be derived as a corollary of a rather general technical result Theorem~\ref{thm:effplig} combined with Theorem~\ref{main:theo}; see Corollary~\ref{cor:properstuff}.

We do not know whether the assumption of effective local compactness can be dropped from Theorem~\ref{main2:theo}. However,  similarly to Theorem~\ref{main:theo},  we do know that the result cannot be improved to give a computable proper metric. This is (essentially) because the aforementioned Corollary~\ref{cor:ce} actually gives a \emph{discrete} example.
Since having a discrete example is perhaps not particularly exciting, in Proposition~\ref{prop:profi} we produce an example of a profinite group that is right-c.e.~Polish and effectively compact, but is not homeomorphic to any effectively compact computable Polish group. Note that in the compact case every metric is automatically proper.

\

\subsection{Connections to other notions in the literature.}\label{sec:further} We now briefly discuss how the notions of effective presentability of Polish groups studied in the paper are related to other notions of effectiveness in the literature.

In Section~\ref{sec:sep} we also illustrate that, in the discrete case, computable Polish presentability is equivalent to computable presentability in the sense of Mal'cev~\cite{Ma61} and Rabin~\cite{abR}. We will also see that right-c.e.~Polish presentability is equivalent to c.e.~presentability for discrete groups.
Note that, in the discrete case, all our presentations  are vacuously effectively locally compact.

Some version of computable (local) compactness seems to be a  necessary extra assumption in a `truly' computable presentation.
As illustrated in \cite{Pontr}, computable Polish presentations do not make Pontryagin duality effective in the compact abelian case.
In contrast, the aforementioned \cite{Pontr} and the recent~\cite{lupini,EffedSurvey} establish effective versions of Pontryagin duality for computably  compact connected abelian groups and  `recursive' profinite abelian groups.
We shall not define `recursive presentations' of profinite groups, since it has been recently discovered in \cite{EffedSurvey} that a profinite group is `recursively presented' if, and only if, it admits a computably compact presentation.
We will use profinite groups and Pontriagin duality in the proof of Proposition~\ref{prop:profi}.
 We  suspect that, for a profinite group, co-c.e.~presentability should perhaps  be equivalent to effectively compact right-c.e.~Polish presentability.
 This is certainly the case for some profinite groups, as exploited (implicitly) in the proof of Proposition~\ref{prop:profi}.

In the satellite paper \cite{separating} we investigate the especially nice case of computably locally compact computable  Polish groups. Quite interestingly, we show that in the totally disconnected locally compact (tdlc) case, a group admits a presentation like that if, and only if, it is computably presentable in the sense of \cite{tdlc,lupini}. We note that \cite{tdlc,lupini} contains several equivalent definitions of computable presentability of a tdlc group, all of which turn out to be equivalent. These equivalent definitions also generalize the  profinite and discrete cases discussed above, and additionally make the Pontryagin - van Kampen duality fully effective for tdlc abelian groups whose duals are also tdlc.

Beyond local compactness,  the  notion of a computable Polish group turned out to be closely related to computable structure theory.
Interestingly, many results in computable structure theory can be viewed as a special case of a computable Polish group computably acting on a computable Polish space. Also, typically the more general result requires a simpler proof; see \cite{MeMo}.
As noted in \cite{MeMo}, many results in \cite{MeMo} can be carried under the weaker assumption that the group is computable topological and admits a c.e.~strong (or formal) inclusion; see Remark~\ref{rem:formal} for a discussion. Quite unexpectedly, the proof of our first main result Theorem~\ref{main:theo}  implies that these seemingly strong extra assumptions in~\cite{MeMo} can be completely dropped when we talk about computable topological Polish groups. See Remark~\ref{rem:formal}  for an explanation.

\

Of course, there are other potential notions of computable presentability that could perhaps work for some special subclasses of Polish groups.
 For instance,  we have mentioned left-c.e.~(lower semi-computable) Polish spaces; these are defined similarly to right-c.e.~Polish spaces, but they seem less well-understood than the latter. For instance, even finding a `natural' example of such a space that would not be obviously computable Polish is a bit of a challenge.
 It is known however that there is a left-c.e.~Polish space not homeomorphic to any computable Polish space~\cite{separating}. Left-c.e.~Polish spaces do not necessarily induce a natural `computable topological' structure, and thus  perhaps are not suitable for representing topological groups in general.  However, interestingly,  every left-c.e.~Polish Stone space is homeomorphic to a computable Polish space~\cite{separating} and, thus, to a computably compact one \cite{uptohom,EffedSurvey}. So it could be that the notion is suitable and well-behaved in the context of profinite or tdlc groups\footnote{The reason behind it is that, while right-c.e.~spaces make formal inclusion c.e., left-c.e.~spaces make formal disjointedness c.e., and thus we can effectively split the space into connected components. We omit these definitions.}.

We already discussed the weak notion of an effectively closed subgroup of $S_\infty$. Other weak notions include
 (hyper)arithmetical presentations of higher degree, such as
 $\Delta^0_\alpha$-Polish and  right- or left-$
 \Sigma^0_\alpha$ Polish presentations. Indeed, we have already mentioned $\Delta^0_\alpha$-Polish presentations above. We strongly conjecture that most of these definitions can be separated from each other by direct relativisation of the known effective results or  using Pontryagin duality and the corresponding results from the discrete abelian case~\cite{Khi}. However, the importance and the exact role of these notions in the theory is not yet clear (beyond their use in extreme counter-examples such as the one in~\cite{uptohom}).

\subsection{The two main definitions} We believe that the results presented in the present paper, combined with various results in \cite{tdlc, Pontr, lupini, separating, EffedSurvey, uptohom, sinf} some of which have been discussed above, establish a solid foundation for the rapidly emerging general theory of algorithmically presented topological groups. In particular, it appears that the basic definitions of effective presentability are robust and nicely align themselves (via direct equivalence or duality) with the  well-established
notions that work for profinite and computable groups.
At least in the important case of locally compact Polish groups, the overall intuition seems to be  as follows:

\smallskip

\begin{center}

`computable Polish' + `computably locally compact'  $\sim$ `computable'

\

 and

\

`right-c.e.~Polish' + `effectively locally compact' $\sim$ `computably enumerable',
\end{center}

\medskip

\noindent where 
$\sim$ stands for `should be viewed as an adequate generalisation of'. The subtle difference between computable compactness and effective compactness
will be elaborated in the preliminaries.

\

 There are many open questions that can be attacked in the new theory; e.g., we cite~\cite{MDsurvey} for problems related to computable classification. For example, which classes of profinite or tdlc groups admit a Friedberg enumeration? What is the complexity of the index set of (say) $SO_3(\mathbb{R})$, up to topological isomorphism? The list goes on.
Also, we wonder if Pointryagin - van Kampen duality works for arbitrary computably locally compact abelian groups;  this question has been raised in \cite{lupini}. We leave these (and many other) questions of this sort open for future investigation.

\section{Preliminaries}

\subsection{Computable topological spaces}
As we already stated in the introduction,
\emph{we assume that all our topological spaces  and groups  are Hausdorff and second-countable.}
The definition below is central to the paper.

\begin{definition}[see, e.g., Definition 2.1 of  \cite{KudKorTop}
of Definition 4 of~\cite{comptop}]\label{def:compttopsp}\rm
	A \emph{computable topological space} is given by a computable, countable basis of its topology for which the intersection of any two basic open sets (``basic balls'') can be uniformly computably listed. More formally, it
	 is a tuple $(X,\tau,\beta,\nu)$ such that
	\begin{itemize}
		
		\item $\beta$ is a base of $\tau$ consisting of non-empty sets,
		
		\item $\nu\colon \omega \to \beta$ is a computable surjective map, and
		
		\item there exists a c.e. set $W$ such that for any $i,j\in\omega$,
		\[
			\nu(i) \cap \nu(j) = \bigcup \{ \nu(k)\,\colon (i,j,k) \in W\}.
		\]
	\end{itemize}
	
	We say that a topological space has a \emph{computable topological presentation} if it is homeomorphic to a computable topological space (with is of course called a computable topological presentation of the space).
	
\end{definition}

Let $(X,\tau,\beta,\nu)$ be a computable topological space. For $i\in\omega$, by $B_i$ we denote the open set $\nu(i)$. As usual, we identify basic open sets $B_i$ and their $\nu$-indices. In order to simplify our notation even further, we will never actually use the notation $(X,\tau,\beta,\nu)$ and will just say that $\tau$ is a computable topological presentation of $X$.

We note that computable topological spaces are closed under taking finite direct products; the computable topology is given by the product topology (equiv., box topology).
We will not need infinite direct products of spaces in the paper.

 Perhaps, the most natural examples of computable topological Polish spaces
are right-c.e.~spaces; see, e.g., Theorem 2.3 of \cite{KudKorTop}; we also cite~\cite{bastone,EffedSurvey} for a detailed proof.
For instance, every computably metrized Polish space is a computable topological space. We discuss these notions in the next subsection.

\begin{definition}\label{def:point} \rm
We call
\[
N^x = \{i: x \in B_{i}\}
\]
\emph{the name of $x$} (in a computable topological space $ X$).
\end{definition}

We can also use basic open balls to produce names of open sets, as follows.

\begin{definition}\label{def:open}\rm
A \emph{name} of an open set $U$ in a computable topological space $X$ is a set  $W \subseteq \mathbb{N}$ such that $U = \bigcup_{i \in W} B_i$, where $B_i$ stands for the $i$-th basic open set in the basis of $X$.
\end{definition}

 If  an open $U$ has a c.e.~name, then we say that $U$ is \emph{effectively open}.
If $C$ is closed then its name is the name of its complement. We say that $C$ is effectively closed, or simply $\Pi^0_1$, if its complement is effectively open.
We also say that a closed set $C$ is $\Sigma^0_1$ if we can list all basic open sets that intersect $C$. A closed set is computable compact if it is both $\Sigma^0_1$ and $\Pi^0_1$.

Recall that an enumeration operator turns enumerations of sets into enumerations of sets and leads to the notion of an enumeration degree.
We omit the formal definition; see ~\cite{rogers}. Enumeration operators are also sometimes called enumeration functionals.
They can be thought of as Turing functionals that use only `positive' information about the oracle.
The notions defined below are also standard.

\begin{definition}\rm\label{def:opencont}
Let $f$ be a map between two computable topological spaces.

\begin{enumerate}
\item We say that $f$ is effectively continuous if there is an enumeration operator $\Phi$ that on input (any enumeration of) a name of an open set $Y$ (in $Y$), lists a name of $f^{-1}(Y)$ (in $ X$).

\item We say that $f$ is effectively open if there is an enumeration operator that given (any enumeration of) a name of an open set $A$ in $X$, lists a name of the open set $f(A)$ in $Y$.

\end{enumerate}
\end{definition}

If $f$ is a homeomorphism, then it is is effectively open if, and only if, $f^{-1}$ is effectively continuous. We thus say that an homeomorphism between two computable topological spaces is effective (or computable) if its is both effectively open and effectively continuous.

One special case of a computable $f$ is the following:

\begin{definition}\rm
If $X$ is a computable topological space then a metric $d$ compatible with the topology of $X$ is computable if there is an enumeration functional $\Phi$ such that for any $y,z\in X$, and any enumeration $p,q$ of $N^y$ and $N^z$ respectively, $\Phi(p\oplus q)$ produces an enumeration of both the left and thew right cuts of the real  $d(y,z)$.

\end{definition}

Equivalently, we could require that the metric is a computable map $X^2 \rightarrow \mathbb{R}$, where $\mathbb{R}$ is equipped with the usual computable topology generated by rational intervals.

\begin{definition}\rm In the notation of the previous definition,
 a metric $d$ is said to be right-c.e.~(upper-semicomputable) if $\Phi(p\oplus q)$ enumerates the right cut of $d(y,z)$. (A left-c.e.~metric is defined similarly.)
\end{definition}

Notice that we do not require the metric to be complete, and we do not require the existence of a computable dense sequence in the space. Of course,  completeness and effective separability are  highly desirable properties, especially because we are mainly interested in Polish(able) spaces. Many natural computable metrics in standard Polish spaces  are complete and, furthermore,  `effectively separable'. We will discuss the complete effectively separable case in Subsection~\ref{Polish:section}.

\subsubsection{Computable topological groups}
We are ready to formally define the notion of a computable topological group.

\begin{definition}\rm\label{comptopgr}
A computable topological group is a triple $(G, \cdot, ^{-1})$, where $G$ is a computable topological space and the group operations
$\cdot: G \times G \rightarrow G$ and  $^{-1}: G \rightarrow G$ are effectively continuous.
\end{definition}

We also say that a topological group has a \emph{computable topological presentation} if it is homeomorphic to a computable topological group. (We call the latter a computable topological presentation of the group.) On order to simplify our notation, we will usually simply say `$G$ is computable topological' rather than  `$(G, \cdot, ^{-1})$ is computable topological'.

\begin{fact}
Multiplication and inverse operators are both effectively open in a computable topological group.
\end{fact}

\begin{proof}
Given some name for an effectively open set $U$, in order to enumerate the name for $U^{-1}$, simply enumerate the preimage of $^{-1}$ on $U$. This must be the name for $U^{-1}$ since $\left(U^{-1}\right)^{-1}=U$. Thus $^{-1}$ is effectively open.

Now given names for open $U,V$, we want to produce computably a name for $U\cdot V$.
The map $(x, y) \rightarrow x^{-1} y$ is computable.
Eenumerate all $V'$ s.t. $$U^{-1}\cdot V'\subseteq V,$$ which is the same as enumerating all $V'$ with the property $V'\subseteq UV$.
\end{proof}

\subsubsection{Realtivization}\label{subs:rel}
We can relativise all definitions in this section. For instance, we can talk about topological presentations that are not necessarily computable.
For instance, given a space,  if we could fix $\tau = \{B_i: \i \in \omega\}$ (that can be identified with $\omega$) and a set $W_\tau$ such that $$B_i \cap B_j = \bigcup \{ B_k \colon (i,j,k) \in W_\tau\}$$
and call it a topological presentation of the space. We stretch our terminology slightly and denote the representation by $\tau$ rather than $W_\tau$.

Then, given such
 topological presentations $\tau_0$ and $\tau_1$, we can define the notions of an effectively continuous and effectively open maps with resect to (w.r.t.)  to  $\tau_0$ and $\tau_1$ by direct relativisation. For instance, we shall need the following relativised notion of effective compatibility of presentation.

\begin{definition}\rm
Topological presentations $\tau_0$ and $\tau_1$ of $X$ are  \emph{effectively compatible} on $X$ if the identity function on $X$ is a homeomorphism that is computable w.r.t.~$\tau_0$ and $\tau_1$. That is, each basic open set in $\mathcal{B}_0$ is effectively open with respect to $\mathcal{B}_1$ and vice versa, uniformly in the indices for the basic open sets.

\end{definition}

\subsection{Effectively metrized spaces and groups}\label{Polish:section}

The notions below are standard.

\begin{definition}\rm
Fix a Polish(able) space $M$.

\begin{enumerate}
\item It is  \emph{computable Polish} or \emph{computably (completely) metrized} if there is a compatible, complete metric $d$ and a countable sequence of \emph{special points} $(x_i)$ dense in $M$ such that, on input $i, j, n$, we can compute a rational number $r$ such that $|r - d(x_i, x_j)| <2^{-n}$.

\item  It is \emph{right-c.e.~Polish} or  \emph{upper-semicomputable Polish} if there is a compatible, complete metric $d$ and a countable sequence of \emph{special points} $(x_i)$ dense in $M$ such that, on input $i, j$, the right cut $\{r \in \mathbb{Q}: d(x_i, x_j) < r \}$ of $d(x_i, x_j)$ can be unformly computably enumerated.

\end{enumerate}

\end{definition}
Clearly, any computable Polish space is right-c.e.~Polish, but there are examples of right-c.e.~Polish spaces that are not even homeomorphic to any computable Polish spaces; \cite{bastone}. We have already mentioned above, every right-c.e.~Polish space can be viewed as a topological space.
Indeed, in a right-c.e.~Polish space define a \emph{basic open ball} to be an open ball having a rational radius and centred in a special point. Let $(B_i)$ be the effective list of all its basic open balls, perhaps with repetition. It is not too difficult to show using triangle inequality that  $(B_i)$ induces a computable topological structure on the space; e.g., \cite{KudKorTop,bastone,EffedSurvey}.  In contrast, in the satellite paper~\cite{separating} we show that there is a Polish(able) space that admits a computable topological presentation, but is not homeomorphic to any right-c.e.~Polish space.

\begin{remark}\rm
Computability of a function between computable (more generally, right-c.e.) Polish spaces admits several reformulations equivalent to Def.~\ref{def:opencont}. For instance, we can require that a map is computable if it uniformly transforms fast converging sequences of special points  to fast converging sequences of special points; see \cite{EffedSurvey,IlKi} for a detailed exposition of computable metric space theory
 (which we omit here). In \cite{MeMo}, such lemmas were formally verified for computable topological spaces with c.e.~`strong (formal) inclusion'; both computable Polish and right-c.e.~Polish spaces have this property.
\end{remark}

Recall that all our topological spaces and groups are Hausdorff and second-countable.
Since every right.c.e.~Polish space is computable topological, and since every computable Polish space is right.-c.e., the following notions are direct generlizations of
Def.~\ref{comptopgr}.

\begin{definition}\rm

\begin{enumerate}

\item A computable Polish group is computable Polish space together with computable group operations $\cdot$ and $^{-1}$.

\item A right-c.e.~Polish group is right-c.e.~Polish space together with computable group operations $\cdot$ and $^{-1}$.

\end{enumerate}

\end{definition}

The notion of a computable Polish group is due to Melnikov and Montalb{\'a}n~\cite{MeMo}.
We will see that in the discrete case it is equivalent to computable presentability of the group in the sense of Mal'cev~\cite{Ma61} and Rabin~\cite{abR}.
We will also see that the notion of a right-c.e.~Polish group can be viewed as a generalization of the standard notion of a c.e.-presented ($\Sigma^0_1$-presented, positive) group in effective algebra~\cite{ErGon}, and of a co-c.e.~presented profinite group
studied by LaRoche  and Smith~\cite{Smith1}.
The latter notion, however, also additionally assumes a certain weak version of \emph{effective compactness} which is another important notion that we discuss next.

We also note that the notions defined in this section can be relativized. For instance, we could define the notion of a Polish presentation in which we do not restrict the computable complexity of the metric.

\subsection{Effective compactness and effective local compactness}
The definition below is standard in the literature:

\begin{definition}\rm\label{def:effcomp} We say that a computable topological space $X$ if \emph{effectively compact} (as a space) if there exists a uniformly computably enumerable list
of all tuples $\langle i_1, \ldots, i_k \rangle$ such that  $X = B_{i_1} \cup B_{i_2} \cup \ldots \cup B_{i_k}$.

\end{definition}

Note that the basic $B_{i}$ are assumed to be non-empty throughout. (If we drop this assumption then we conjecture that we get a weaker notion.)
In the context of a computable topological space, one could also talk about effective compactness of a subset of the space.

\begin{definition}\label{def:effcomp1}\rm We say that a compact subset $K$ of computable topological space $X$ if \emph{effectively compact} (as a subset) if there exists a uniformly computably enumerable list
of all tuples $\langle i_1, \ldots, i_k \rangle$ such that  $K \subseteq B_{i_1} \cup B_{i_2} \cup \ldots \cup B_{i_k}$.

\end{definition}

The following fact is well-known (for example, see \cite[Lemma 2.3]{xu_grubba}):

\begin{fact}\label{lem:comsub}
Given a name for a closed set $A$ and a name for a compact set $K$, we can list a name for $A\cap K$.
%
\end{fact}

\begin{remark}\rm\label{strange:remark}
 We will get a stronger notion of a \emph{computably compact} set if, in Def.~\ref{def:effcomp1}, we additionally require that $B_{i_j} \cap K \neq \emptyset$ for every cover, i.e., that the set is $\Sigma^0_1$ closed.
Of course, this would not be an issue in Def.~\ref{def:effcomp} because each basic open $B_{i_j}$ is non-empty and, thus, vacuously `intersects' the space.
The difference is already seen in the context of $2^{\omega}$, where a compact set is  effectively compact (as a subset) if, and only if, it is~$\Pi^0_1$.
Recall also that a non-empty $\Pi^0_1$ class does not have to contain any computable points at all.
In other words, even in $2^{\omega}$ we get that effective compactness of $K$ implies only that the open complement of $K$ can be listed without necessarily making the set `computably closed'. (Recall that a closed set $C$ is computably closed iff its complement is c.e.~open and also we can list all basic open $B$ such that $B \cap C \neq \emptyset$.)
In other words, an effectively compact subset of a computable Polish space does not have to be `computably compact' = `effectively compact' + `computably closed'.
Indeed, it has been shown in \cite{bastone}, there is a $\Pi^0_1$-class in $2^{\omega}$ that is not even \emph{homeomorphic} to any computable Polish space (thus, to any effectively compact computable Polish space). Nonetheless, there is a very nice correspondence between
effectively compact \emph{spaces} and computably compact \emph{sets}, as essayed in \cite{EffedSurvey,IlKi}.
\end{remark}

We shall adopt:

\begin{center}
\begin{quote}
When we talk about effectively compact subsets of a computable topological space, we mean Def.~\ref{def:effcomp1}
and do not necessarily assume the set is additionally $\Sigma^0_1$ closed.
\end{quote}
\end{center}

\begin{definition}[E.g., \cite{xu_grubba}]\rm \label{def:compact}
The name of a compact subset $K$ of a computable topological space $X$ is the
list of all finite covers  of $K$ by basic open sets of $X$.
\end{definition}

More formally, the name $N^K$ of $K$ is the set $$\{\langle i_0, \ldots, i_k \rangle: K \subseteq \bigcup_{j \leq k} B_{i_j}\},$$
where $B_i$ is the $i$-th basic open set in $X$.
Then evidently $K$ is effectively compact if, and only if, it has a c.e.~name. (Note that we do not require that each $B_{i_j}$ has to intersect the set.)

\subsubsection{Effective local compactness}
Unlike the notion of an effectively compact space which is robust \cite{EffedSurvey,IlKi}, there are several (seemingly) non-equivalent notions of effective local compactness in the literature
\cite{Pauly,xu_grubba,Weih_Zheng}. We shall not attempt to verify whether these notions of effective local compactness are equivalent up to homeomorphism since there is yet not enough evidence that these notions are equally important and iseful. We adopt the following:

\begin{definition}[\cite{Pauly}]\rm
A second countable topological space $X$ is effectively locally compact if there is an enumeration operator which, given a name for a point $x\in X$ and a name for an open set $U\ni x$, lists a name for an open set $V$ and a name for a compact set $K$ such that $x\in V\subseteq K\subseteq U$.
\end{definition}

As explained in the companion paper~\cite{separating}, in the context of computable Polish spaces,
we can drop $U$ in the definition above and assume $K$ is a computably compact (closed) ball of an arbitrary small
computable radius. However, groups in the present paper are rarely computable Polish.
Nonetheless, later in the paper we will be dealing with computable topological groups that additionally have a dense sequence of computable points in them.
The following lemma shall be useful:

\begin{lemma}[{\cite[Proposition 8]{Pauly}}]\label{lem:paulyercs}\rm
Suppose that a computable topological space $X$ has a dense set of uniformly computable points. Then $X$ is effectively locally compact if and only if there is a triple $\left( \{U_n\}_{n\in\omega},\{K_m\}_{m\in\omega},R \right)$, where
\begin{itemize}
\item $U_n$ is a  computable sequence of (uniformly) effectively open sets.
\item $K_m$ is a computable sequence of (uniformly) effectively compact sets.
\item $R\subseteq\mathbb{N}\times\mathbb{N}$ is a c.e. set such that $(n,m)\in R\Rightarrow U_n\subseteq K_m$.
\item For any open set $U$, we have
\[U=\bigcup_{\{m\mid K_m\subseteq U\}}\bigcup_{\{n\mid (n,m)\in R\}}U_n.\]
\end{itemize}

\end{lemma}

In other words, under the assumptions of the lemma, any open set can be essentially approximated by compact neighbourhoods from within, with a sufficient degree of effectiveness.
The triple $\left( \{U_n\}_{n\in\omega},\{K_m\}_{m\in\omega},R \right)$ was called an ercs for $X$ in \cite{Pauly}. Compare this to the related notion of a \emph{computably locally compact Hausdorff space} in \cite[Definition 3.2]{xu_grubba}.

\subsection{Effectively proper metrics}\label{sec:properdef}

Recall that a metric $d$ is proper if every closed bounded ball $\{y\mid d(x,y)\leq r\}$ is compact; equivalently, every closed bounded set is compact. Recall that a name of a closed set is a name of its open complement, and the name of a compact set is the list of all of it finite basic open covers.
The obvious effectivization of properness  is the following:

\begin{definition}\rm
A right-c.e.~Polish space $(\mathcal{M},d)$ is effectively proper if there exists an enumeration operator which, given a name  a closed set $A$ and a  basic open ball $B_d(\alpha,r)\supseteq A$,  lists a name for some compact set $K\supseteq A$.
%
%
\end{definition}

The lemma below relates the notion with local effective compactness.

\begin{lemma}
Given a Polish space $(\mathcal{M},d)$, we have (i) $\Leftrightarrow$ (ii) $\Rightarrow$ (iii). If $(\mathcal{M},d)$ is a computable Polish space then all three are equivalent:
\begin{itemize}
\item[(i)] $(\mathcal{M},d)$ is effectively proper.
\item[(ii)] Given an enumeration of a name for a closed set $A$ and a basic open ball containing $A$, we can list a compact name for $A$.
\item[(iii)] Given a (closed) name for the closed ball $B^\leq_d(\alpha,r) = \{x: d(\alpha, x) \leq r\}$ and the parameters $\alpha,r$, we can compute a compact name for $B^\leq_d(\alpha,r)$.
\end{itemize}
\end{lemma}
\begin{proof}
(i) $\Rightarrow$ (ii): Apply Fact \ref{lem:comsub}, which holds even if $d$ is not computable. The implications (ii) $\Rightarrow$ (i),(iii) are trivial. If $d$ is computable then assuming (iii) holds, for each basic open basic $B(\alpha,r)$ where $\alpha$ is a special point and $r\in\mathbb{Q}^+$, we have $B_d^{\leq}(\alpha,r)$ is effectively closed.
(This is because a special point $x \notin B_d^{\leq}(\alpha,r)$ together with a $q$-ball $B_d^{<}(x, q)$ iff $ d(\alpha, x)> r+q$, which is c.e.~if the metric is computable, but could be not c.e.~for a right-c.e. metric. In the case of a computable metric we do not have to assume $B^\leq_d(\alpha,r)$ comes together with a closed name, as it is automatically can be reconstructed from its parameters.)
 So given a name for a closed set $A$ and some $\alpha,r$ such that $A\subseteq B_d(\alpha,r)$ we can obtain a name for $B^{\leq}_d(\alpha,r)$ and, therefore, a compact name for $B^{\leq}_d(\alpha,r)$. By Fact \ref{lem:comsub}, we obtain a compact name for $A$.
\end{proof}

The first item in the above lemma is the effective version of the fact that every closed and bounded set is contained in a compact set. The second item corresponds to the fact that every closed and bounded set is compact, while the third item corresponds to the fact that every closed and bounded ball is compact.
An effectively locally compact computable Polish space will satisfy a version of $(iii)$ that says that, for every $x$, \emph{there exists} a sufficiently small $r$ and a special $\alpha$ such that $x \in B^\leq_d(\alpha,r)$ is effectively compact. In \cite{separating} we additionally show that, in such a space, $ B^\leq_d(\alpha,r)$ can be picked computably closed as well.
But of course, to claim that we can compute  a compact name  for $B^\leq_d(\alpha,r)$ for \emph{any} $r, \alpha$ the metric has to be proper at the first place.

\subsection{A unified generalization: represented spaces}
The definition of a compact name Def.~\ref{def:compact} is reminiscent of Def.~\ref{def:point} for points. It is also  somewhat similar to Def.~\ref{def:open} for open sets, but unlike $N^x$ and $N^K$, an open set will typically have lots of names, not just one name.
Also, the definitions of a computable topological,  a computable Polish, and a right-c.e.~Polish space have a similar flavour too. More specifically, in each case  we can define  names of points and define the notion of a computable map between presentations.

All these notions have some clear similarities and seem to be special instances of something more general.
This intuition can be made formal using the theory of \emph{represented spaces}. We do not need this degree of generality in the present paper.
Indeed,  one of the main goals of our paper is to illustrate that, in the case of Polish groups, we can safely restrict ourselves to the classical theory of effectively metrized spaces without any loss of generality.
We cite \cite{ArnoSurvey} for a detailed exposition of the theory of represented topological spaces.

\section{Effective Birkhoff-Kakutani Theorem}\label{Sec:ebk}
The classical Birkhoff-Kakutani Theorem is the following:
\begin{theorem}[Birkhoff-Kakutani Theorem]\label{thm:bkthm}
Let $G$ be a topological group. Then $G$ is metrizable iff $G$ is Hausdorff and first countable. Moreover, if $G$ is metrizable, then $G$ admits a compatible left-invariant metric.
\end{theorem}

We consider the effective version of the Birkhoff-Kakuani Theorem and restrict attention to computable topological groups.
We obtain the following effective version of Theorem \ref{thm:bkthm}.

\begin{theorem}\label{thm:effbk}
Let $G$ be a computable topological group. Then $G$ admits a right-c.e.~compatible left-invariant metric. \end{theorem}

\begin{proof}
Let $\{B_{n}\}_{n\in\omega}$ be the effective basis for $G$, where each $B_{n}$ is nonempty.
We extract a local base for $e_{G}$ in the following way. Define
\begin{itemize}
    \item $\U_{0}=G$,
    \item $\U_{n+1}=\U_{n}\cap\left(B_{n}\cdot B_{n}^{-1}\right)$ for $n\geq 0$.
\end{itemize}

Note that each $\U_{n}$ is effectively open, uniformly in $n$. First we check that $\{\U_{n}\}_{n\in\omega}$ is a local base for $e_{G}$. Consider the continuous function $f(x,y)=xy^{-1}$. Then for any open set $U$ containing $e_{G}$, there must be some basic open ball $B_{k}$ s.t.~$f(B_{k},B_{k})\subseteq U$, since $x\cdot x^{-1}=e_{G}$ for any $x\in G$. Thus $\{B_{n}\cdot B_{n}^{-1}\}_{n\in\omega}$ gives a local base for $e_{G}$. Since each $B_n\neq\emptyset$, hence $e_G\in B_{n}\cdot B_{n}^{-1}$ for every $n$ and so it follows that $\{\U_{n}\}_{n\in\omega}$ is also a local base for $e_{G}$.

Next we check that $\U_{n}=\U_{n}^{-1}$ for all $n$. Base case $n=0$ is trivially true, then assume $\U_{n}$ is true for some $n$. Let $x\in \U_{n+1}=\U_{n}\cap\left(B_{n}\cdot B_{n}^{-1}\right)$. By inductive hypothesis, we know that $x^{-1}\in \U_{n}$. Furthermore, since $\left(B_{n}\cdot B_{n}^{-1}\right)^{-1}=\left\{\left(xy^{-1}\right)^{-1}\mid x,y\in B_{n}\right\}=\left\{yx^{-1}\mid x,y\in B_{n}\right\}= B_{n}\cdot B_{n}^{-1}$, then $x^{-1}\in B_{n}\cdot B_{n}^{-1}$, and thus $x^{-1}\in \U_{n+1}$.

Since $G$ is Hausdorff, for any $x\neq e_G$ there is an open set $U$ such that $e_G\in U$ and $x\not\in U$. There is some $n$ such that $\U_n\subseteq U$ which means that $x\not\in \U_n$. So this means that $\bigcap_{n\in\omega}\U_{n}=\{e_{G}\}$.

Now we define $\{\V_{n}\}_{n\in\omega}$ satisfying the following properties for every $n\in\omega$:
\begin{enumerate}
    \item $\V_{0}=G;\V_{n+1}\subseteq \V_{n};$\label{1}
    \item $\V_{n}=\V_{n}^{-1};$\label{2}
    \item $\V_{n+1}^{3}\subseteq \V_{n};$\label{3}
    \item $\V_{n}\subseteq \U_{n}$.\label{4}
\end{enumerate}
Consider the function $f(x_{0},y_{0},x_{1},y_{1},x_{2},y_{2})=\Pi_{i=0}^{2}x_{i}y_{i}^{-1}$. Assume that $\V_n$ has been defined. Search for basis elements $B$ and $B_{i_0},\cdots, B_{i_5}$ satisfying the following
\begin{itemize}
    \item $\langle B_{i_{j}}\rangle_{j<6}\subseteq f^{-1}(\V_{n})$, and
    \item $B\subseteq \V_{n}\cap\bigcap_{j=0}^{5}B_{i_{j}}$.
\end{itemize}
Since $G$ is a computable topological group, multiplication and inverse are effectively continuous operations, and hence, given an index for $\V_{n}$, we can search for an index for
$B$. Take $\V_{n+1}=(B\cdot B^{-1})\cap \U_{n+1}$, and note that $\V_{n+1}$ is also effectively open, uniformly in $n$. Note also that $B$ must exist, since $e_G\in \V_n$ and $f(e_G,e_G,\cdots,e_G)=e_G$.

The properties \ref{2} and \ref{4} are immediate. By choice of $B$, we see $f(B,B,\dots,B)\subseteq \V_{n}$, and thus $\V_{n+1}^{3}=\{x\cdot y\cdot z\mid x,y,z\in (B\cdot B^{-1})\cap \U_{n+1}\}\subseteq\{x\cdot y\cdot z\mid x,y,z\in (B\cdot B^{-1})\}=f(B,B,\dots,B)\subseteq \V_{n}$. Finally, since $e_{G}\in \V_{n+1}$, then for any $x\in \V_{n+1}$, $x\cdot e_{G}\cdot e_{G}\in \V_{n+1}^{3}$, which gives $\V_{n+1}\subseteq \V_{n+1}^{3}\subseteq \V_{n}$.

Now we define the functions $\varrho,d:G^2\rightarrow \mathbb{R}$ by:
\begin{itemize}
    \item $\varrho(x,y)=\inf\{2^{-n}\mid x^{-1}y\in \V_{n}\}$.
    \item $d(x,y)=\inf\{\sum_{i=0}^{l}\varrho(g_{i},g_{i+1})\mid g_{i}\in G, g_{0}=x,g_{l+1}=y,l\in\omega\}$.
\end{itemize}
Since $\V_0=G$, $\varrho$ and $d$ are total functions.

To verify that $d$ is a metric on $G$, we follow the classical proof (see \cite{GaoSu}) almost exactly. Since for each $n$, $\V_{n}=\V_{n}^{-1}$, then we have that $\varrho(x,y)=\varrho(y,x)$ for any $x,y\in G$, and since $\varrho(gx,gy)=\inf\{2^{-n}\mid (gx)^{-1}(gy)\in \V_{n}\}=\inf\{2^{-n}\mid x^{-1}y\in \V_{n}\}=\varrho(x,y)$, then $d$ must also be both symmetric and left-invariant. It is easy to see that $d(x,x)=0$, and since $\varrho(g,h)=0$ for any $g,h\in G$, then $d(x,y)\geq 0$ for any $x,y\in G$. From the definition of $d$, it is also easy to see that the triangle inequality holds. Thus it remains to check that $d(x,y)=0$ only if $x=y$.

We prove the following by induction
\begin{equation*}
    \sum_{i=0}^{l}\varrho(g_{i},g_{i+1})\geq\frac{1}{2}\varrho(g_{0},g_{l+1})
\end{equation*}
for any $g_{0},g_{1},\dots,g_{l+1}\in G$. First note that by property \ref{3} of $\{\V_{n}\}_{n\in\omega}$, $\varrho$ has the following property
\begin{equation}\tag{*}\label{*}
    \forall\varepsilon>0,\text{ if }\varrho(g_{0},g_{1}),\varrho(g_{1},g_{2}),\varrho(g_{2},g_{3})\leq\varepsilon,\text{ then }\varrho(g_{0},g_{3})\leq 2\varepsilon.
\end{equation}
Then when $l\leq 2$, the proposition follows directly from (\ref{*}). Suppose that the proposition holds for all $l'<l$ for some $l$. Let $S=\sum_{i=0}^{l}\varrho(g_{i},g_{i+1})$, and $m$ be the largest (possibly $m=1$) s.t. $\sum_{i=0}^{m-1}\varrho(g_{i},g_{i+1})\leq\frac{1}{2}S$. By the inductive hypothesis, $\varrho(g_{0},g_{m})\leq 2\sum_{i=0}^{m-1}\varrho(g_{i},g_{i+1})\leq S$. Since $\sum_{i=0}^{m}\varrho(g_{i},g_{i+1})>\frac{1}{2}S$, then $\sum_{i=m+1}^{l}\varrho(g_{i},g_{i+1})\leq\frac{1}{2}S$, then by inductive hypothesis again, we have $\varrho(g_{m+1},g_{l+1})\leq S$. But clearly $\varrho(g_{m},g_{m+1})\leq S$, that is $\varrho(g_{0},g_{m}),$ $\varrho(g_{m},g_{m+1}),\varrho(g_{m+1},g_{l+1})\leq S$, and hence by (\ref{*}), $\varrho(g_{0},g_{l+1})\leq 2S$. Then if $d(x,y)=0$, it must be that $\frac{1}{2}\varrho(x,y)=0$ and this is only the case when $x^{-1}y=e_{G}$, i.e. $x=y$.

Now we check that $d$ is compatible with the topology of $G$. Let $U$ be open in $G$ and $g\in U$. Then for some $n\in\mathbb{N}$, $g\V_{n}\subseteq U$. We check that $B_{d}(g,2^{-n-1})\subseteq U$. Let $h\in B_{d}(g,2^{-n-1})$, then $d(h,g)<2^{-n-1}$. By the claim above, $\varrho(g,h)\leq 2d(g,h)<2^{-n}$, and by the definition of $\varrho$, $g^{-1}h\in \V_{n}$, thus $h\in g\V_{n}\subseteq U$. Conversely, let $U$ be open in the topology given by $d$ and let $g\in U$. For some $n\in\mathbb{N}$, we have that $B_{d}(g,2^{-n})\subseteq U$. We check that $g\V_{n+1}\subseteq U$. Let $h\in g\V_{n+1}$, then $\varrho(g,h)\leq 2^{-n-1}$, and by definition of $d$, $d(g,h)\leq \varrho(g,h)\leq 2^{-n-1}<2^{-n}$, therefore $h\in B_{d}(g,2^{-n})\subseteq U$.

Now we check that the right cut of $d(x,y)$ can be enumerated given an enumeration of $N^x$ and $N^y$.
List out all finite tuples, and search for $l+1$-tuples $\langle p_{m}\rangle_{m\leq l}$, $\langle q_{m}\rangle_{m\leq l}$ and $\langle n_{m}\rangle_{m\leq l}$ such that
the sequence $(B_{p_{0}},B_{p_{1}},\dots,B_{p_{l}})$ and $(B_{q_{0}},B_{q_{1}},\dots,B_{q_{l}})$ satisfy
\begin{itemize}
    \item $B_{p_{0}}$ is enumerated in $N^{x}$,
    \item $B_{q_{l}}$ is enumerated in $N^{y}$,
    \item $B_{p_{m}}^{-1}\cdot B_{q_{m}}\subseteq\V_{n_{m}}$ for each $m\leq  l$, and
    \item $B_{q_{m}}\cap B_{p_{m+1}}\neq\emptyset$ for each $m<l$.
\end{itemize}
Note that the third condition is c.e. and implies that  $(B_{p_{m}},B_{q_{m}})\subseteq\varrho^{-1}\left([0,2^{-n_{m}}]\right)$. If a suitable $\langle n_{m}\rangle_{m<l}$ is found, enumerate $\sum_{j=0}^{l}2^{-n_{j}}$ into our approximation to the right cut of $d(x,y)$.


Now we verify that the procedure described above in fact does enumerate the right cut of $d(x,y)$. Let $q=\sum_{j=0}^{l}2^{-n_{j}}$ be enumerated by the procedure at some stage. This means that some $\langle n_{j}\rangle_{j\leq l}$ and corresponding $(B_{i_{0}},B_{i_{1}},\dots,B_{i_{l}})$ and  $(B_{k_{0}},B_{k_{1}},\dots,B_{k_{l}})$ have been found where $x\in B_{i_{0}}$ and $y\in B_{k_{l}}$. Now for each $j<l$, let $g_{j}\in B_{k_{j}}\cap B_{i_{j+1}}$. Then we have:
\begin{itemize}
    \item $x^{-1}g_{0}\in \V_{n_{0}}$,
    \item $g_{j}^{-1}g_{j+1}\in \V_{n_{j+1}}$ for each $j<l-1$, and
    \item $g_{l-1}^{-1}y\in \V_{n_{l}}$.
\end{itemize}
Since $d(x,y)=\inf\{\sum_{i=0}^{l}\varrho(g_{i},g_{i+1})\mid g_{i}\in G, g_{0}=x,g_{l+1}=y,l\in\omega\}$, so we have $d(x,y)\leq \sum_{i=0}^{l}\varrho(g_{i},g_{i+1})\leq\sum_{j=0}^{l}2^{-n_{j}}=q$.

Now to check that the procedure enumerates some rational $q$ s.t $q<d(x,y)+\varepsilon$ for each $\varepsilon>0$. We can assume that $x\neq y$, since if $x=y$, it can be easily seen that the procedure described before enumerates the right cut of $0$. Let $g_{i}\in G$ for $i\leq l+1$ where $g_{0}=x,g_{l+1}=y$ be given s.t. $\sum_{i=0}^{l}\varrho(g_{i},g_{i+1})<d(x,y)+\varepsilon$, for some $\varepsilon>0$. We may of course assume that $\varrho(g_{i},g_{i+1})>0$ for each $i$, since $\varrho(g_{i},g_{i+1})=0$ iff $g_{i}=g_{i+1}$. Hence $\varrho(g_{i},g_{i+1})=2^{-n_{i}}$ for some $n_{i}$. At some stage  $(B_{i_{0}},B_{i_{1}},\dots,B_{i_{l}})$ and $(B_{k_{1}},B_{k_{2}},\dots,B_{k_{l}})$  must be found satisfying the conditions above and we enumerate $\sum_{i=0}^{l}2^{-n_{i}}$ into the right cut of $d(x,y)$. Thus the procedure enumerates rationals arbitrarily close to $d(x,y)$.
\end{proof}


\section{Computing a dense sequence}\label{Sec:ebkpoints}

In Theorem \ref{thm:effbk} we produced a compatible right-c.e.~metric for any given Hausdorff computable topological group. However the effectivization was point-free, in the sense of lacking a countable dense subset of points. Perhaps unexpectedly, if we \emph{assume} the metric that we produce is actually complete, then we can
show that the metric admits a dense computable sequence of points. In other words, in this case we obtain a right-c.e.~Polish presentation of the group.

\begin{theorem}\label{thm:completemet}\rm
Let $G$ be computable topological group where the metric $d$ produced in Theorem \ref{thm:effbk} is complete. Then $G$ has  a right-c.e.~Polish presentation.\end{theorem} 

Most of the rest of the section is devoted to the proof of the theorem.  We begin with several technical lemmas that establishes several useful properties of the metric produced in the proof of Theorem~\ref{thm:effbk}.

\subsection{Two technical lemmas}


Let $\mathcal{M}=(\{\alpha_i\}_{i\in\omega},d)$ be a countable metric space and $\overline{\mathcal{M}}$ be its completion. Let $\tau_d$ be the topology on $\overline{\mathcal{M}}$ generated by the metric $d$ with basis elements $B_d(\alpha_i,\varepsilon)$ where $i\in\omega$ and $\varepsilon\in\mathbb{Q}^+$.
Note these balls also form a base of the restricted topology on $\mathcal{M}$.
We say that
a computable topological space $G$ (we mainly care about topological groups) 
 is  \emph{effectively compatible} with $\mathcal{M}$ if $\mathcal{M}\subseteq G\subseteq \overline{\mathcal{M}}$ and where $\tau$ and $\tau_d$ (restricted as the subspace topology) are effectively compatible on $G$.

 In the definitions above,
we do not require $\tau_0$ or $\tau_1$ to be computable topologies, even though our topologies will typically be computable; we discussed relativization in \S\ref{subs:rel}. Also, we  do not restrict the complexity of $d$, even though we are interested in right-c.e.~metrics which induce a computable topology. The technical reason why we need this extra degree of generality will be explained in Remark~\ref{rem:afterlemma} shortly. Recall that all our groups are Hausdorff.



\begin{lemma}\label{lemma:effcom}\rm
Let $(G,\tau,\mathcal{B})$ be a computable 
topological group and let $d$ be the metric produced in Theorem \ref{thm:effbk}. Suppose that $G$ contains a dense set of points $\{\alpha_{i}\}_{i\in\omega}$ w.r.t.~$\tau$, and there is a computable function $\varphi$ such that for every $i,s$ we have $d(\alpha_{i},g)\leq 2^{-s}$ for any $g\in B_{\varphi(i,s)}\in\mathcal{B}$. Then $G$ is effectively compatible with $(\{\alpha_i\}_{i\in\omega},d)$.

\end{lemma}

\begin{remark}\rm\label{rem:afterlemma}
Even though $d$ is a right-c.e.~point-free metric defined on $G$, however, since $\{\alpha_i\}_{i\in\omega}$ need not be computable points w.r.t.~$\tau$, it is not immediately obvious why $d(\alpha_i,\alpha_j)$ is right-c.e.~uniformly in $i,j$.~This is in fact true, but we will not need it here (yet). We shall revisit this  later. 
\end{remark}

\begin{proof}[Proof of Lemma~\ref{lemma:effcom}]
Since $\{\alpha_{i}\}_{i\in\omega}$ is dense with respect to $\tau$, it is also dense in $G$ with respect to $d$ since $\tau$ and $\tau_d$ are compatible as shown in Theorem \ref{thm:effbk}. Thus $G\subseteq \overline{(\{\alpha_i\}_{i\in\omega},d)}$.

Given $B_{d}(\alpha_{i},r)$, where $r\in\mathbb{Q}^{+}$, we want to effectively (in $i$ and $r$) produce a name for $U=B_{d}(\alpha_{i},r)$ with respect to  $\tau$. For each $s$ where $2^{-s}<r$, list all finite tuples and search for $\langle n_{m}\rangle_{m\leq l},\langle p_{m}\rangle_{m\leq l}$ and $\langle q_{m}\rangle_{m\leq l}$ s.t. the sequences $(B_{p_{0}},B_{p_{1}},\dots,B_{p_{l}})$ and $(B_{q_{0}},B_{q_{1}},\dots,B_{q_{l}})$ satisfying
\begin{itemize}
    \item $B_{p_{m}}^{-1}\cdot B_{q_{m}}\subseteq\V_{n_{m}}$ for each $m\leq l$,
    \item $B_{q_{m}}\cap B_{p_{m+1}}\neq\emptyset$ for each $m<l$,
    \item $B_{p_{0}}\cap B_{\varphi(i,s)}\neq\emptyset$, and
    \item $2^{-s}+\sum_{m\leq l}2^{-n_{m}}<r$.
\end{itemize}
If such sequences are found, enumerate 
$B_{q_{l}}$ into the name for $U$.

We follow a similar argument in the proof of Theorem \ref{thm:effbk}. It is clear that whenever a sequence satisfying the above properties is found, we have the property that for any $g\in B_{q_{l}}$, $d(g,\alpha_{i})<r$. Thus $B_{q_l}\subseteq U$. Conversely, for any $g\in U=B_{d}(\alpha_{i},r)$, there must be some sequence which witnesses that $d(\alpha_{i},g)<r$, then at some stage we must find it and thus enumerate a set containing $g$ into the name of $U$. Thus $U$ is effectively open with respect to $\tau$.

Now given a basic open set $B_j\in\tau$, we produce a c.e. name for $B_j$ w.r.t. $\tau_{d}$. Consider the function $f(x,y,z)=x\cdot(y\cdot z^{-1})$. Since $f$ is effectively continuous w.r.t. $\tau$, we can effectively enumerate $f^{-1}(B_j)$. Search for $\tau$-basic open sets $X,Y,Z$ s.t.
\begin{itemize}
    \item $Y\cap Z\neq\emptyset$,
    \item $X\cap B_{\varphi(i,s)}\neq\emptyset$ for some $i$, and $s>n+2$ where $n$ is found so that $B_{n}\subseteq Y\cap Z$, and
    \item $X\cdot Y\cdot Z^{-1}\subseteq B_j$.
\end{itemize}
For each $X,Y,Z$ and $i,s,n$ found satisfying the above, we enumerate the ball $B_{d}(\alpha_{i},2^{-n-1}-2^{-s})$ into the name for $B_j$ w.r.t. $\tau_{d}$.

Suppose $X,Y,Z,i,s,n$ are found by the procedure above. Since $B_{n}\subseteq Y\cap Z$, we have that $X\cdot(B_{n}\cdot B_{n}^{-1})\subseteq B_j$, and hence for any $g\in X$, $g\V_{n}\subseteq B_j$, since $\V_{n}\subseteq B_{n}\cdot B_{n}^{-1}$. Since $X\cap B_{\varphi(i,s)}\neq\emptyset$, we can fix $g\in X\cap B_{\varphi(i,s)}$, and so $d(\alpha_{i},g)\leq 2^{-s}<2^{-n-2}$. As a result, for any $h\in B_{d}(\alpha_{i},2^{-n-1}-2^{-s})$, $d(g,h)\leq d(g,\alpha_{i})+d(\alpha_{i},h)<2^{-s}+2^{-n-1}-2^{-s}=2^{-n-1}$, that is $h\in B_{d}(g,2^{-n-1})$. However from the classical proof of compatibility (see the proof of Theorem \ref{thm:effbk}), we know that since $g\V_{n}\subseteq B_j$, then $B_{d}(g,2^{-n-1})\subseteq B_j$. Therefore we conclude that $B_{d}(\alpha_{i},2^{-n-1}-2^{-s})\subseteq B_j$.

Now conversely fix some $g\in B_j$. There are $\tau$-basic open sets $X,Y,Z$ such that $X\cdot Y\cdot Z^{-1}\subseteq B_j$ such that $g\in X$ and $e\in Y\cap Z$. Now fix any $n$ (such that $B_n\subseteq Y\cap Z$). Since $\tau$ and $\tau_d$ are compatible, and $\{\alpha_{i}\}_{i\in\omega}$ is dense in $G$ wrt $\tau_d$, we can pick some $i$ and $s>n+2$ such that $B_d(\alpha_i,2^{-s+1})\subseteq X$. (We may also assume that $d(\alpha_i,g)<2^{-s}$). But this means that $B_{\varphi(i,s)}\subseteq X$.
%
%
These six items $X,Y,Z,i,s,n$ must be thus found by the above procedure. Furthermore, since $d(\alpha_{i},g)<2^{-s}<2^{-n-1}-2^{-s}$, then $g\in B_{d}(\alpha_{i},2^{-n-1}-2^{-s})$.

Hence the procedure above witnesses that $B_j$ is effectively open wrt $\tau_d$. Thus, $\tau$ and $\tau_d$ are effectively compatible.
\end{proof}

\begin{lemma}\label{fact1}Let $G$ be a computable topological group that contains a dense set of uniformly computable points. Then $G$ is effectively compatible with a right-c.e.~metric space. Furthermore, the compatible metric (on $G$) is also left-invariant.
\end{lemma}
\begin{proof}
Apply Theorem \ref{thm:effbk} to produce a compatible right-c.e. metric $d$ for $G$. Let $\{\alpha_i\}_{i\in\omega}\subseteq G$ be the set of uniformly computable points with respect to the original topology $\tau$ on $G$. Now $d(\alpha_i,\alpha_j)$ is right-c.e.~uniformly in $i,j$ by following the procedure in Theorem \ref{thm:effbk} and feeding to the procedure the c.e. names for $N^{\alpha_i}$ and $N^{\alpha_j}$.  Now take $\mathcal{M}=(\{\alpha_i\}_{i\in\omega},d)$.

It remains to check that $(G,\tau)$ is effectively compatible with $\mathcal{M}$. By Lemma \ref{lemma:effcom}, we need only check that there is a computable function $\varphi(i,s)=j$ s.t. $d(\alpha_{i},g)\leq 2^{-s}$ for any $g\in B_{j}$. (Here $B_j$ are the basic open sets of $\tau$). Given $i,s$, search for some basic open set $B_j$ such that $B_j\subseteq X\cap Y$, $\alpha_i\in B_j$ and $X^{-1}\cdot Y\subseteq \mathcal{V}_s$.
Some $B_j$ must be found since $\alpha_{i}^{-1}\cdot\alpha_{i}=e_{G}\in\V_{s}$. Now define $\varphi(i,s)=j$. Since $B_{j}^{-1}\cdot B_{j}\subseteq\V_{s}$, it must be that $\forall g,h\in B_{j}$, $d(g,h)\leq\varrho(g,h)\leq 2^{-s}$, in particular, $\forall g\in B_{j}$, $d(\alpha_{i},g)\leq 2^{-s}$.
\end{proof}

By effective compatibility, the group operations remain computable with respect to the metric. Thus, to argue that the group from Theorem~\ref{thm:completemet} has a computable Polish presentation, all we need to show is that there is a computable dense sequence.

\begin{remark}\label{rem:formal}\rm We can easily reconstruct a computable dense sequence in a computable topological group with  c.e.~formal inclusion. The latter is an  axiomatic generalisation of formal inclusion in metric spaces that is defined as follows:
$$B(x, q) \subset_{form} B(y, r) \mbox{ iff } d(x,y)+q < r.$$
Note it is a c.e.~relation  in a right-c.e.~space.
 We omit the definition of \emph{abstract} formal inclusion $\gg$ and refer the reader to \cite{MeMo}, but we note that  \cite{MeMo} contains an example of a  non-metrisable computable topological space
 with c.e.~strong inclusion.
We claim that Lemma \ref{fact1} implies the following:
\emph{Every  computable topological group with a c.e.~formal inclusion is effectively compatible with a right-c.e.~metric space.}
To see why, extract a dense set of uniformly computable points by considering for each $i_0$  the first found effective sequence $B_{i_0}\gg B_{i_{1}}\gg\dots$ where $\bigcap_k B_{i_k}=\{\alpha_{i_0}\}$. It follows from  the definition of formal inclusion in \cite{MeMo} that $N^{\alpha_{i_0}}$ is c.e.; we omit the details.
\end{remark}

\subsection{Proof of  Theorem \ref{thm:completemet}}
In order to identify special points in $G$, we use some ideas in \cite{Grubba} by utilising the group operations. 
 Fix $G$ and $d$ as in the proof of Theorem~\ref{thm:effbk}. The first lemma below is designed to implement the idea sketched in Remark~\ref{rem:formal}, but in the absence of
c.e.~`formal inclusion'.

\begin{lemma}\label{point-free}
Given a basic open set $B_{i}$ in $G$, there exists a computable sequence of basic open sets $\{B_{i_{s}}\}_{s\in\omega}$ such that:
\begin{enumerate}
    \item $B_{i_0}=B_i$ and $B_{i_{s+1}}\subseteq B_{i_{s}}$ for every $s$.
    \item $dm\left(\overline{B_{i_{s}}}\right)\coloneqq\sup\left\{d(x,y)\mid x,y\in \overline{B_{i_{s}}}\right\}\leq 2^{-s}$ for every $s$.
    \item $f^{*}\left(B_{i_{s}},B_{i_{s}}\right)\subseteq \V_{s}$ for every $s$, where  $f^{*}(x,y)=x^{-1}y$.
\end{enumerate}
\end{lemma}

Notice that we only claim that the sequence $\{B_{i_{s}}\}_{s\in\omega}$ is computable (uniformly in $i=i_0$). We make no claims about how difficult it might be to approximate $dm\left(\overline{B_{i_{s}}}\right)$.
\begin{proof}
Let $B_{i_{0}}=B_{i}$. Since $\V_0=G$, hence $d(x,y)\leq 1$ for every $x,y\in G$ and so properties 2 and 3 are trivially true for $s=0$. Now suppose inductively that $B_{i_{s}}$ satisfying the desired properties has been defined. Since $f^{*}$ is both effectively open and effectively continuous, and since $e_G\in f^{*}\left(B_{i_{s}},B_{i_{s}}\right)\cap \V_{s+1}$, we can search for some basic open set $B$ s.t. $f^{*}(B,B)\subseteq f^{*}(B_{i_{s}}, B_{i_{s}})\cap\V_{s+1}$ and $B\cap B_{i_{s}}\neq\emptyset$.
Take $B_{i_{s+1}}$ to be any basic open set contained in $B\cap B_{i_{s}}$. This gives property 1. Note that $f^{*}\left(B_{i_{s+1}},B_{i_{s+1}}\right)\subseteq \V_{s+1}$ which gives property 3.

It remains to check property 2. Let $x,y\in\overline{B_{i_{s}}}$ be given. Since the original topology on $G$ is compatible with the topology on $G$ induced by the metric $d$, we can find sequences $(x_{n})_{n\in\omega}$ and $(y_{n})_{n\in\omega}$ such that $x_{n},y_{n}\in B_{i_{s}}$ and $d(x_n,x), d(y_n,y)\leq 2^{-n}$ for all $n$.
Therefore, $\forall\varepsilon>0$, $\exists m,n$ for which $d(x_{n},x)<\frac{\varepsilon}{2}$ and $d(y_{m},y)<\frac{\varepsilon}{2}$, then
\begin{align*}
    d(x,y)&\leq d(x,x_{n})+d(x_{n},y_{m})+d(y_{m},y)\\
    &\leq \frac{\varepsilon}{2}+2^{-s}+\frac{\varepsilon}{2}\quad (\text{since } f^{*}\left(B_{i_{s}},B_{i_{s}}\right)\subseteq\V_{s})\\
    &\leq 2^{-s}+\varepsilon
\end{align*}
Since the above holds for any $\varepsilon$, then $d(x,y)\leq 2^{-s}$.
\end{proof}

Since $d$ is complete and compatible with the original topology, by property 3 of Lemma \ref{point-free}, there is a unique point $\alpha_{i}\in G$ s.t.~$$\alpha_{i}\in\bigcap_{s\in\omega}\overline{B_{i_{s}}}.$$ Here $\alpha_i$ corresponds to the sequence $\{B_{i_{s}}\}_{s\in\omega}$ with $B_{i_0}=B_i$. Repeating this for all $i$ produces an infinite sequence $\alpha_0,\alpha_1,\cdots$ such that $\alpha_i\in\overline{B_i}$ for each $i$. When we want to distinguish between the different sequences $\{B_{i_{s}}\}_{s\in\omega}$ we will use the notation $\{B^k_{i_{s}}\}_{s\in\omega}$ if $B^k_{i_0}=B_k$.



We claim that
the set $\{\alpha_{i}\}_{i\in\omega}$ produced in Lemma \ref{point-free} is dense in the original topology of $G$.

To see why,
fix a basic open set $B_{i}$ be given and let $g\in B_{i}$. Since $G$ is metrizable, $G$ is also (classically) regular. Then $\exists F\ni g$ s.t. $F\subseteq B_{i}$ and $F$ is a closed neighbourhood of $g$. Since the basic open balls form a basis for the topology of $G$, $\exists B_{j}\subseteq F$, and hence $\overline{B_{j}}\subseteq F\subseteq B_{i}$. Then we must have that $\alpha_{j}\in\overline{B_{j}}\subseteq B_{i}$.

\

We now finish the proof of the theorem.
Since $\{\alpha_{i}\}_{i\in\omega}$ is dense in $\tau$ and using the procedure in Lemma \ref{point-free}, given any $i,s$, we are able to effectively identify a basic open set $B_{j}$ for which $\forall g\in B_{j},\,d(\alpha_{i},g)\leq 2^{-s}$. By Lemma \ref{lemma:effcom}, we have that $\tau$ is effectively compatible with $\tau_{d}$.

First of all notice that $d(\alpha_{i},\alpha_{j})$ is a right-c.e.~real uniformly in $i,j$. To see this, note that an easy modification of the right~c.e.~approximation procedure of $d$ in Theorem \ref{thm:effbk} by requiring that $B_{p_{0}}\cap B_{\varphi(i,s+1)}\neq\emptyset$ and $B_{q_{l}}\cap B_{\varphi(j,s+1)}\neq\emptyset$ and enumerating $2^{-s}+\sum_{m\leq l}2^{-n_{m}}$ allows us to produce the right cut of $d(\alpha_{i},\alpha_{j})$. But since $\{\alpha_i\}_{i\in\omega}$ are obviously uniformly computable points with respect to $\tau_d$, and since $\tau$ and $\tau_d$ are effectively compatible, and also that $d$ is right-c.e., we conclude that $\{\alpha_i\}_{i\in\omega}$ are also uniformly computable points with respect to~$\tau$.
%

\subsection{Consequences of Theorem \ref{thm:completemet}.}


 A topological space will be called topologically complete if it admits a metric under which it is complete.
 Recall also that a metric is invariant if it is both left and right invariant. For instance, every left- or right-invariant metric
 in an abelian group is invariant.
  Klee~\cite{Klee} proved the following result.
 Suppose G is a group with invariant metric $d$. If the space $(G, d)$ is topologically complete, then $G$ is actually complete under $d$.
 We therefore obtain the following, rather satisfying:

 \begin{corollary}\label{cor:abelian}
Let $G$ be a computable topological group that is Polish(able) abelian. Then $G$ is effectively compatible with a right-c.e.~Polish group.
\end{corollary}

\noindent Of course, the metric is invariant in this case.

\

Since all compact metric spaces are complete, by Theorem~\ref{thm:completemet} we see that any compact computable topological group must be compatible with a right-c.e.~metric space. In fact, we will se that the same is true of locally compact groups; this fact will be established at the end of the next section.

\section{Locally compact groups and proper metrizations}\label{proper:section}
Recall that a metric $d$ is proper if every closed bounded ball $\{y\mid d(x,y)\leq r\}$ is compact; equivalently, every closed bounded set is compact.
As we have already mentioned above, Struble showed the following:

\begin{theorem}[Struble \cite{Struble}, see also  \cite{Uffe}]\label{thm:Struble}
Let $G$ be a topological group which is Hausdorff, second countable and locally compact. Then $G$ admits a compatible left-invariant proper metric.
\end{theorem}
We refer the reader to \cite{Struble} for the classical proof.
This section is devoted to proving that the following effective version of Theorem~\ref{thm:Struble}.
(Recall that all our groups are Hausdorff and second countable.)

\begin{theorem}\label{thm:effplig}
Let $(G,\tau)$ be an effectively locally compact computable topological group. Then $G$  is effectively compatible with an effectively proper right-c.e. metric space. Furthermore the metric is left-invariant.
\end{theorem}

It will also follow from the construction that if the metric produced in Theorem \ref{thm:completemet} is complete, then the proper right-c.e.~metric is also complete (cf.~the proof of Corollary~\ref{cor:locallycompact}).

\begin{proof}
 By (the proof of) Corollary \ref{cor:locallycompact} and Theorem \ref{thm:completemet}, $(G,\tau)$ contains a dense set of uniformly computable points $\{\alpha_i\}_{i\in\omega}$. By Lemma \ref{fact1}, $G$ is effectively compatible with $(\{\alpha_i\}_{i\in\omega},\delta)$ where $\delta$ is right-c.e. Furthermore $\delta$ is left-invariant.

By Lemma \ref{lem:paulyercs} we fix the triple $\left( \{B_n\}_{n\in\omega},\{K_m\}_{m\in\omega},R \right)$. Note that each $B_n$ is $\tau$-effectively open (uniformly in $n$) and collectively form a basis for $(G,\tau)$. We may also assume that for every $n$ there is some $m$ such that $(n,m)\in R$. By set product and power we mean the corresponding operation with respect to the group operation.

Note that the group identity $e$ is (in any computable topological group) a computable point with respect to $\tau$, therefore there is some $\tau$-effectively compact set $K$ and some $\tau$-open set $B$ such that $e\in B\subseteq K$. Since $\delta$ is compatible with $\tau$ we fix some $r\in\mathbb{Q}^+$ such that $B_\delta(e,r)\subseteq B\subseteq K$. By scaling $\delta$ we can assume that $r=2$, and so we may assume that $B_\delta(e,2)\subseteq K$. Note that we do not claim that $\overline{B_\delta(e,2)}$ or $B^{\leq}_\delta(e,2)$ is effectively compact, merely that some open ball around $e$ is contained in an effectively compact set $K$.



We now define a collection $\{U_{r}\}_{r\in\mathbb{Q}^{+}}$ of $\tau$-effectively open sets satisfying the following properties.
\begin{enumerate}
    \item For each $r\in\mathbb{Q}^{+}$, $U_{r}$ is contained in some $\tau$-effectively compact set.
    \item For each $r\in\mathbb{Q}^{+}$, $U_{r}=U_{r}^{-1}$.
    \item For each $r,s\in\mathbb{Q}^{+}$, $U_{r}\cdot U_{s}\subseteq U_{r+s}$.
    \item $\forall r<2,\,U_{r}=B_{\delta}(e,r)$.
    \item $\bigcup_{r\in\mathbb{Q}^{+}}U_{r}=G$.
    \item For each $r\in\mathbb{Q}^{+}$, $e\in U_{r}$.
\end{enumerate}
For $0<r<2$, define $U_{r}=B_{\delta}(e,r)$. To check that the properties hold, since $\delta(g,e)=\delta(g^{-1},e)$ for any $g\in G$, $U_{r}$ is closed under inverse for $r<2$. For $r+s<2$, let $x\in U_{r}$ and $y\in U_{s}$, then by triangle inequality and left-invariance of $\delta$, $\delta(xy,e)\leq\delta(xy,x)+\delta(x,e)=\delta(y,e)+\delta(x,e)$, thus giving that $xy\in U_{r+s}$.

Now we define $U_{2}=B_{\delta}(e,2)\cup W_{2}$, where $W_{2^{n}}$ is defined as follows. For each $n$, take $W_{2^{n}}=B_{n}\cup B_{n}^{-1}$. Then $W_{2^n}$ is $\tau$-effectively open (uniformly in $n$) and closed under inverse. If $B_{n}\subseteq K_{m}$, then by the effective continuity of $^{-1}$, $K_{m}^{-1}\supseteq B_{n}^{-1}$ is also effectively compact. We get that $U_{2}$ is contained in $K\cup K_{m}\cup K_{m}^{-1}$ which are effectively compact. It is then clear then that we have $\{U_{r}\}_{r\leq 2}$ with the desired properties.

Suppose inductively that $U_{r}$ for $r\leq 2^{n}$ have been defined s.t. each $U_{r}$ is $\tau$-effectively open and $U_{2^{n}}$ is contained in some $\tau$-effectively compact set. For each $2^{n}<r<2^{n+1}$, list out all finite sequences of positive rationals $\langle t_{i}\rangle_{i\leq m}$ s.t. $t_{i}\leq 2^{n}$ for each $i$ and $\sum_{i\leq m}t_{i}=r$. For each such sequence listed out, enumerate $\prod_{i\leq m}U_{t_{i}}$ into the open name of $U_{r}$. By inductive hypothesis, since each $U_{t_{i}}$ is effectively open, and multiplication is effectively open, then $U_{r}$ must also be effectively open (uniformly in the index $r$). Finally take $U_{2^{n+1}}=W_{2^{n+1}}\cup \left(U_{2^{n}}\cdot U_{2^{n}}\cdot U_{2^{n}}\cdot U_{2^{n}}\right)$.

To see that property 3 holds, if $r+s<2^{n+1}$, then the desired property follows easily from the definition of $U_{r}$. Suppose then that $r+s=2^{n+1}$. If $r=s=2^{n}$, then $U_{r}\cdot U_{s}=U_{2^{n}}\cdot U_{2^{n}}\subseteq U_{2^{n+1}}$ (note that $e\in U_{2^n}$). Therefore we may assume that $r>2^{n}$ and $s<2^{n}$. Then for any sequence $\langle t_{i}\rangle_{i\leq m}$ where $\sum_{i\leq m}t_{i}=r$, $\exists m_{0},m_{1}$ s.t. $\sum_{i=0}^{m_{0}}t_{i}\leq 2^{n}$, $\sum_{i=m_{0}+1}^{m_{1}}t_{i}\leq 2^{n}$ and $\sum_{i=m_{1}+1}^{m}t_{i}\leq 2^{n}$. By inductive hypothesis, we have that $\prod_{i=0}^{m_{0}}U_{t_{i}},\prod_{i=m_{0}+1}^{m_{1}}U_{t_{i}},\prod_{i=m_{1}+1}^{m}U_{t_{i}}\subseteq U_{2^{n}}$. Thus this gives us that $U_{r}\subseteq\left(U_{2^{n}}\right)^3$. Then note that $U_{s}\subseteq U_s\cdot U_{2^n-s}\subseteq U_{2^{n}}$ (again by IH and the fact that $e\in U_{2^n-s}$), and thus $U_{r}\cdot U_{s}\subseteq\left(U_{2^{n}}\right)^4\subseteq U_{2^{n+1}}$.

To check that property 1 holds, note that if $r<2^{n-1}$ then $U_r\subseteq U_r\cdot U_{2^{n+1}-r}\subseteq U_{2^{n+1}}$ by property 3 above, and so it is enough to check that $U_{2^{n+1}}$ is contained in an effectively compact set. By IH, $U_{2^{n}}$ is contained in some effectively compact set $K^{*}$, so $U_{2^{n+1}}$ is contained in $K_{m}\cup K_{m}^{-1}\cup\left(K^{*}\right)^4$, where $m$ is s.t. $(n,m)\in R$. It is not hard to check that $\left(K^{*}\right)^4$ is effectively compact, and hence
%
$U_{2^{n+1}}$ is contained in some effectively compact set.

From the definition of $U_{r}$, for any $r$ where $2^{n}<r\leq 2^{n+1},\,U_{r}=U_{r}^{-1}$ and so property 2 holds as well.

Finally, since $\{B_{n}\}_{n\in\omega}$ is a basis for $(G,\tau)$, we have $\bigcup_{n}W_{2^{n}}=G$ and hence $\bigcup_{r}U_{r}=G$.

Now we define the metric $d$ on $G$ by $d(x,y)=\inf\{r\mid x^{-1}y\in U_{r}\}$. To see that $d$ is a metric, note that $d(x,y)=0$ gives that $\forall 0<r<2, x^{-1}y\in U_{r}=B_{\delta}(e,r)$, meaning that $\delta(x^{-1}y,e)=0$. Since $\delta$ is a metric, it has to be that $x=y$. By property 6, $d(x,x)=0$. The symmetry of $d$ and triangle inequality follow from property 2 and 3 of $\{U_{r}\}_{r\in\mathbb{Q}^{+}}$ respectively. $d$ is obviously left-invariant. It remains to check that $(G,d,\{\alpha_{i}\}_{i\in\omega})$ is a right-c.e. metric space, $G$ is effectively compatible with $(G,d,\{\alpha_{i}\}_{i\in\omega})$, and that $d$ is effectively proper. First of all, we have:

\begin{lemma}\label{lem:d_delta_equivalent}
For all $x,y\in G$, if $d(x,y)<2$ or $\delta(x,y)<2$ then $d(x,y)=\delta(x,y)$.
\end{lemma}
\begin{proof}If $d(x,y)<2$ then $d(x,y)=\inf\{r<2\mid x^{-1}y\in U_r\}=\inf\{r<2\mid \delta(x,y)<r\}=\delta(x,y)$. If $\delta(x,y)<2$ then $x^{-1}y\in U_r$ for some $r<2$, which means that $d(x,y)<2$ and so by the above, $d(x,y)=\delta(x,y)$.
\end{proof}

Recall that the sequence $\{\alpha_i\}_{i\in\omega}$, apart from being used as special points for $d$ and $\delta$, are also uniformly computable points with respect to $\tau$. Then together with the fact that $U_r$ are $\tau$-effectively open (uniformly in the index $r$), one can obviously give a right-c.e. approximation to $d(\alpha_i,\alpha_j)$, uniformly in $i,j$.

%
%
By Lemma \ref{lem:d_delta_equivalent}, we have $\overline{(\{\alpha_{i}\}_{i\in\omega},d)}=\overline{(\{\alpha_{i}\}_{i\in\omega},\delta)}\supset G$, so it is sufficient to show that $\tau_d$ and $\tau_\delta$ are effectively compatible on $G$.
Let $B_{d}(\alpha_{i},r)$ be given. Since $d$ is right-c.e., $\ll_d$ is a c.e. relation, where $\ll_d$ is the usual formal inclusion relation for basic metric balls. Consider the $\tau_\delta$-effectively open set consisting of all $B_{\delta}(\alpha_{j},q)$ such that $q<2$ and $B_{d}(\alpha_{j},q)\ll_d B_{d}(\alpha_{i},r)$. This shows that $B_{d}(\alpha_{i},r)$ is $\tau_\delta$-effectively open. To show that each $B_{\delta}(\alpha_{i},r)$ is $\tau_d$-effectively open is similar.


Finally to check that $d$ is effectively proper, we note that by definition of $d$, $B_{d}(e,r)\subseteq U_{r}\subseteq U_{2^{n}}$ for some sufficiently large $n$. Given a closed set $F$ and an open ball $B_{d}(\alpha_{i},q)\supseteq F$, take $r=d(\alpha_{i},e)[0]+q$, where $d(\alpha_{i},e)[0]$ is the first rational enumerated by the right cut of $d(\alpha_{i},e)$, then note that $F\subseteq B_{d}(\alpha_{i},q)\subseteq B_{d}(e,r)\subseteq U_{r}$. For any $n>\log_{2}(r),\,F\subseteq U_{2^{n}}\subseteq K$, where $K$ is $\tau$-effectively compact. But this means that $K$ is also $\tau_{d}$-effectively compact and a compact name can be found uniformly in $n$.
\end{proof}

\subsection{Consequences}
In the first corollary, we do not assume that the group is effectively locally compact. The first corollary uses   Struble's original result  and  Theorem \ref{thm:completemet}, and the second corollary follows from  Theorem~\ref{thm:effplig} and Theorem~\ref{thm:completemet}.

\begin{corollary}\label{cor:locallycompact}
Let $G$ be a computable topological group that is locally compact. Then $G$ is effectively compatible with a right-c.e.~Polish group.
\end{corollary}
\begin{proof}
Let $\delta$ be the left-invariant right-c.e.~compatible metric produced in Theorem \ref{thm:effbk}. Struble (see  \cite{Struble}) used $\delta$ to produce another metric $d$ on $G$ such that $d$ is compatible with $\delta$ and $d$ is a proper metric. However (see Lemma \ref{lem:d_delta_equivalent}) $d$ and $\delta$ are equal whenever $d(x,y)<2$ or $\delta(x,y)<2$. Since every proper metric space is complete, this means that $(G,\delta)$ is complete. By Theorem \ref{thm:completemet}, $G$ is effectively compatible with a right-c.e. metric space.
\end{proof}

We see that the metrics $\delta$ and $d$ are equi-complete and effectively compatible assuming that the group is effectively locally compact, by Theorem~\ref{thm:effplig}.
 Thus, by Theorem \ref{thm:completemet} combined with Theorem~\ref{thm:effplig}, we have:

\begin{corollary}\label{cor:properstuff}
If $G$ is an  effectively locally compact computable topological group. Then $G$ admits a right-c.e.~Polish presentation in which the metric is (effectively) proper and left-invariant.
\end{corollary}
In particular, effective local compactness (and the effective compatibility of $\delta$ and $d$ in the notation above) implies the proper metric in the corollary above is also effectively locally compact. But of course, being effectively proper is nicer than just being  effectively locally compact.

\

A natural question arises whether we can strengthen these corollaries further and additionally assume that the metric is computable in each of the corollaries above.
In the next section we show that the answer is `no' in both cases. In fact, our counter-examples corresponding to Corollaries~\ref{cor:locallycompact} and~\ref{cor:properstuff}
are compact and discrete, respectively.

\section{Comparing and separating the notions}\label{sec:sep}

In Sections  \ref{Sec:ebkpoints} and \ref{proper:section} we produced (left-invariant) right-c.e.~Polish presentations of locally compact groups. We now aim to show that this is tight, i.e.~we show that  there are computable topological (Polish) groups which are not effectively compatible with any computable metric space.
In this section we give two examples, one discrete and one profinite. In the process of proving the results we will establish several lemmas that are perhaps more valuable (or interesting) than the counter examples. 

\subsection{Discrete groups} Recall that a computable presentation of a discrete countable group is its isomorphic copy of the form $F/ H$, where $F$ is the standard decidable presentation of the free group upon omega generators, and $H$ is its computable normal subgroup~\cite{abR,Ma61}.  If $H$ is merely c.e., then we say that the group is `c.e.-presented'. (These correspond to `recursive' groups with solvable and not necessarily solvable Word Problem, respectively.) We can pick representatives in each class and assume that the domain of a computable group is $\mathbb{N}$; then the group operations are computable (as functions on $\mathbb{N}$). In the c.e.-presented case, we have to also introduce a computably enumerable congruence on $\mathbb{N}$, but we can still keep the operations computable. The difference is that two elements can be at some stage declared equal.
Note that this is very similar to the difference between computable and right-c.e.~Polish presentations of a group. This intuition is made formal below.


\begin{lemma}\label{lem:discretecomp}
A countable discrete group is computably presentable iff it
admits a computable Polish presentation. 
\end{lemma}

\begin{proof}
Suppose that a group $G$ is computably presentable, i.e. $G$ is generated by $\{\alpha_i\}_{i\in\omega}$ on which the group operations and the equality relation are computable. We consider the standard discrete metric defined on the elements of the computable presentation of $G$, i.e. $d(\alpha_i,\alpha_j)=0$ iff $\alpha_i=\alpha_j$ and $d(\alpha_i,\alpha_j)=1$ otherwise. Since testing of equality is computable by assumption, the metric is also computable.

To check that $\cdot$ is effectively continuous with respect to $\tau_d$, given $\alpha_k\in G$ and $r\in\mathbb{Q}^{+}$, if $r> 1$, we simply enumerate $G\times G$ as the preimage. If $r\leq 1$, then find all pairs $\alpha_i,\alpha_j$ such that $\alpha_i\cdot\alpha_j=\alpha_k$ and enumerate $B_{d}(\alpha_{i},1)\times B_{d}(\alpha_{j},1)$. To see that $^{-1}$ is effectively continuous, it is similar.

Conversely suppose that $G=\{\alpha_i\}_{i\in\omega}$ is a countable group and there is a computable discrete metric $d$ defined on $G$ in which the group operations are effectively continuous with respect to $\tau_d$. Even though $d$ is computable, the isolating radius for each $\alpha_i$ might not be. Nonetheless we (not effectively) fix some rational $r>0$ for which $B_d(e_G,r)$ which isolates $e_{G}$. Since the metric is computable, we can decide given any $\alpha_{i}$, whether or not $e_{G}=\alpha_{i}$, by computing $d(e_{G},\alpha_{i})$ to an accuracy of $\frac{r}{2}$. Therefore equality in $G$ is computable if we can show that the group operations are computable.

To compute $\alpha_{i}^{-1}$, enumerate the preimage of $B_{d}(e_{G},r)$ under $\cdot$ and wait for $(B_{d},B_{d}')$ to show up where $\alpha_i\in B_d$.
Then the center of $B_d'$ is necessarily the inverse of $\alpha_{i}$, as $B_{d}(e_{G},r)$ isolates $e_G$.
Now given $\alpha_{i},\alpha_{j}$, to compute $\alpha_{i}\cdot\alpha_{j}$, search for three basic metric balls $B_d,B_d',B_d''$ such that $B_d\cdot B_d'\cdot (B_d'')^{-1}\subseteq B_{d}(e_{G},r)$ and where $\alpha_i\in B_d$ and $\alpha_j\in B_d'$. Then the center of $B_d''$ is necessarily equal to $\alpha_i\cdot\alpha_j$.
%
%
\end{proof}

\begin{lemma}\label{lem:discretece}
A countable discrete group is c.e. presentable iff it admits a right-c.e.~Polish presentation.

\end{lemma}

\begin{proof}
Suppose that a discrete group $G$, is c.e. presentable, i.e. $G$ is generated by $\{\alpha_i\}_{i\in\omega}$ on which the group operations are computable but the equality relation is c.e. We consider the standard discrete metric defined on the elements of the computable presentation of $G$, i.e. $d(\alpha_i,\alpha_j)=0$ iff $\alpha_i=\alpha_j$ and $d(\alpha_i,\alpha_j)=1$ otherwise. Since equality is c.e. by assumption, the metric is right-c.e.
%
%
To see that the operations are effectively continuous w.r.t. the topology induced by $d$, repeat the same procedure as in Lemma \ref{lem:discretecomp}.

Conversely suppose that $G=\{\alpha_i\}_{i\in\omega}$ is a countable group and there is a right-c.e. discrete metric $d$ defined on $G$ in which the group operations are effectively continuous with respect to $\tau_d$. Again we fix some rational $r>0$ such that $B_d(e_G,r)$ which isolates $e_{G}$. To see that the group operations are computable we follow exactly as in Lemma \ref{lem:discretecomp}, noting that the predicate ``$\alpha_i\in B_d(\alpha_j,r)$'' is still c.e. Since the metric is right-c.e., and the operations are computable, equality in $G$ is c.e.
%
%
%
%
%
%
\end{proof}

\begin{corollary}\label{cor:ce}
There exists a computable topological discrete abelian group  (thus, right-c.e.~Polish)
that is not topologically isomorphic to any computable Polish group.

\end{corollary}

\begin{proof}
Consider the group $G=\bigoplus_{i\in S}\mathbb{Z}_{p^{k}}$ where $S$ is a $\Sigma_{2}^{0}$ set that is not c.e.. Then $G$ is c.e.~presentable \cite{Khas,MelSurvey} with no computable presentation. By Lemma~\ref{lem:discretece} there is a discrete right-c.e.~metric $d$ on $G$ in which the group operations are effectively continuous. Since $d$ is right-c.e., $\tau_d$ is a computable topology on $G$ and so $(G,\tau_d)$ is a computable topological group. If $(G,\tau_d)$ is effectively compatible with the computable metric space $\mathcal{M}=(\{\alpha_i\}_{i\in\omega},d')$ then
$G=\{\alpha_i\}_{i\in\omega}$ since $d'$ is discrete, and since $\tau_d$ and $\tau_{d'}$ are effectively compatible, the group operations would be effectively continuous with respect to $\tau_{d'}$, which contradicts  Lemma \ref{lem:discretecomp}.
\end{proof}

\subsection{A profinite counterexample}
Recall that in Corollary~\ref{cor:properstuff} we produced a right-c.e.~proper Poliosh presentation which, by effective compatibility, was also effectively locally compact.
Can we produce a \emph{computable} (proper) metric, say, in the simplest compact case? Note that in the case of a compact Polish group we vacuously have a proper metric.
 We now prove that the answer is `no'.

\begin{proposition}\label{prop:profi}
There exists a profinite group $G$ that admits an effectively compact right-c.e.~Polish presentation but has no computably compact  (effectively compact computable Polish) presentation.
\end{proposition}

\begin{proof}
The proof that we outline below resembles similar  counter-examples in~\cite{Smith1}, \cite{Pontr}, and \cite{uptohom}. However, in our case a bit more care is needed.

\

We construct our group $G$  to be the inverse limit of finite groups and (injective) homomorphisms
$$F_0 \leftarrow_{\phi_0}  F_1 \leftarrow_{\phi_1} F_2 \leftarrow_{\phi_2} \ldots,$$
where the maps are not necessarily surjective, and the groups and the maps are given by their strong indices, i.e., as finite sets.
Let $H_i$ be the image of $\phi_i$. We have that $$G = \projlim_{i \in \omega}  (H_i, \phi_i).$$

We construct $H$ to be isomorphic to the direct product of cyclic groups $$G_S = \prod_{i \in S} \mathbb{Z}_{p_i},$$
where $S \subseteq \omega$.

\begin{lemma}
 If $G$  has an effectively compact,  computable Polish (eccp) presentation iff $S$ is c.e.
\end{lemma}

\begin{proof}[Proof of Lemma]
By \cite{EffedSurvey}, for $G$ to be eccp it is necessary and sufficient that $G$ has  `recursive' presentation in the sense of Smith~\cite{Smith1}. By
\cite{Pontr}, it is also equivalent to computable presentability of the discrete group $\bigoplus_{i \in S} \mathbb{Z}_{p_i}$, which is also evidently  equivalent to $S$ being c.e. \end{proof}

By the lemma above, it is sufficient to construct an effectively compact right-c.e.~Polish (ecrp) presentation of $G_{\overline{K}}$, where $K$ is the halting set.
A straightforward injury-free construction can be designed to implement the following idea.

\

 In the notation above, we make each $F_i$ to be the product of some of the $\mathbb{Z}_{p_j}$ -- which exactly depends on the construction.
Note we can change our mind about $H_i$ by making $\phi_i$ not onto. We can also delay this decision and make $\phi_j$ not onto for some $j>i$, with a similar effect to the projective limit. Simply put, whenever we introduce another cyclic summand, we then later can `kill' it if necessary, but we also would like to do it in the most natural way possible so that we do not upset the operations.

This is done as follows.
If $i$ enters $K$ at stage $s$, we make sure that $H_s \leq F_s$ is isomorphic to $\prod_{j<i; j \notin K_s} \mathbb{Z}_{p_j}$.
In other words, we pick $H_{s}$ to be the subgroup of $F_s$ of this form and define $\phi_s$ to be (essentially) the identity map that identifies  $F_{s+1}$ with $H_s$ inside $F_{s}$, which clearly can be done. We then re-introduce cyclic summands of orders $>i$ in $H_{s+2}$ and define the map $\phi_{s+1}$ so that respects the group operations, etc.

 We believe that the formal construction is so elementary that the  formal details can be safely left to the reader.
It shall be useful to view this process described above as follows. We declare that the subgroup of $F_s$ generated by the cyclic summands $\mathbb{Z}_{p_j}$, $j\geq i$,  is the kernel of the natural homomorphism $\psi_s$
taking $F_s$ to $F_{s+1} = \prod_{j<i; j \notin K_s} \mathbb{Z}_{p_j}$, where we additionally pick nice representatives of the classes in the factor. These `nice' representatives are the elements having zero $\mathbb{Z}_{p_j}$-projection    if $ \mathbb{Z}_{p_j}$ does not occur in  $F_{s+1}$.
We use these representatives to define $\phi_s$ as induced by  the natural isomorphism between the $F_{s+1}$ and $F_s / U_s $, where $U_s = Ker \, \psi_s$.

We now observe that $\prod_{i \in \omega} F_i$  is computably homeomorphic to $2^{\omega}$.
We essentially declare a basic clopen set in  $\prod_{i \in \omega} F_i$ to be `out' of the class representing the group if
the respective element is discovered to be outside the range of some $\phi_i$. In other words, the domain of the group can be viewed as a $\Pi^0_1$ class.
It is easy to see that a $\Pi^0_1$-calls can be right-c.e.~metrized; e.g., \cite{bastone}. (The special points will be just the special points of $2^{\omega}$, but they can be declared equal.) However, we also need to verify that the group operations are effective.

\begin{remark}\rm
If the reader is puzzled as to what could potentially be an issue, they should think about the following. One of the equivalent formulations of a computable map is that a fast Cauchy sequence should be mapped to a fast Cauchy sequence, uniformly effectively and consistently. However,  what could be \emph{not} fast converging  in $\prod_{i \in \omega} F_i$
can become fast converging after the `collapse' of the metric. We must be able to to still define the operations for such sequences even though we cannot predict what happens in the future.
However,  the `collapse' happens uniformly and symmetrically throughout the group in the construction, so we should be safe.
\end{remark}

The group operations are consistently defined and are total on the whole of $\prod_{i \in \omega} F_i$.
Also, the clopen sets that are declared `out' together with the elements that stay `in' make up finitely many cosets
of $F_s/U_s$,
$$a_0 + U_s, \ldots, a_k + U_s, $$
where $U_s = Ker \, \psi_s $ (as explained above) and the elements $a_k$ can be identified with the respective elements of $F_{s+1}$, under $\psi_i$ (equivalently, with their pre-images under $\phi_i$).
Collapse the whole coset $a_j +U_s$ into one point by declaring that the metric between any two points in the coset is equal to zero.
In $\prod_{i \leq s} F_i$, the operation is consistently defined modulo $U_s$. Thus, if we declare the (new distance between) the elements making up each of these cosets
equal to zero it will be consistent with the group operations.

It should be clear that the metric is right-c.e., and that we can define a computable sequence starting with the standard computable dense sequence in $$\prod_{i \in \omega} F_i \cong 2^{\omega}.$$ Of course, our dense sequence will have lots of repetitions.
The resulting $\Pi^0_1$ class is also effectively compact, as all $\Pi^0_1$ classes in $2^{\omega}$  are, since any cover remains a cover after the metric is being re-defined.

To calculate the operations, use the operations inherited from $\prod_{i \in \omega} F_i \cong 2^{\omega}$. As we argued above, the functionals act consistently with the group structure on the $\Pi^0_1$ class.
 \end{proof}

\bibliographystyle{plain}
\bibliography{topgroups}

\def\cprime{$'$} \def\cprime{$'$} \def\cprime{$'$} \def\cprime{$'$}
\begin{thebibliography}{10}

\bibitem{AshKn}
C.~Ash and J.~Knight.
\newblock {\em Computable structures and the hyperarithmetical hierarchy},
  volume 144 of {\em Studies in Logic and the Foundations of Mathematics}.
\newblock North-Holland Publishing Co., Amsterdam, 2000.

\bibitem{bastone}
Nikolay Bazhenov, Matthew Harrison-Trainor, and Alexander Melnikov.
\newblock Computable stone spaces, 2021.

\bibitem{Boone:59}
W.~Boone.
\newblock The word problem.
\newblock {\em Annals of Math}, 70:207--265, 1959.

\bibitem{Brattka.Hertling.ea:08}
V.~Brattka, P.~Hertling, and K.~Weihrauch.
\newblock A tutorial on computable analysis.
\newblock In {\em New computational paradigms}, pages 425--491. Springer, New
  York, 2008.

\bibitem{metric1}
G.~S. Ceitin.
\newblock Algorithmic operators in constructive complete separable metric
  spaces.
\newblock {\em Dokl. Akad. Nauk SSSR}, 128:49--52, 1959.

\bibitem{EffedSurvey}
R.~Downey and A.~Melnikov.
\newblock Effectively compact spaces.
\newblock Preprint.

\bibitem{MDsurvey}
Rodney~G. Downey and Alexander~G. Melnikov.
\newblock Computable analysis and classification problems.
\newblock In Marcella Anselmo, Gianluca~Della Vedova, Florin Manea, and Arno
  Pauly, editors, {\em Beyond the Horizon of Computability - 16th Conference on
  Computability in Europe, CiE 2020, Fisciano, Italy, June 29 - July 3, 2020,
  Proceedings}, volume 12098 of {\em Lecture Notes in Computer Science}, pages
  100--111. Springer, 2020.

\bibitem{ErGon}
Y.~Ershov and S.~Goncharov.
\newblock {\em Constructive models}.
\newblock Siberian School of Algebra and Logic. Consultants Bureau, New York,
  2000.

\bibitem{Ershov}
Yu. Ershov.
\newblock Problems of solubility and constructive models [in russian].
\newblock 1980.
\newblock Nauka, Moscow (1980).

\bibitem{feiner1970}
Lawrence Feiner.
\newblock Hiearchies of boolean algebras.
\newblock {\em J. Symbolic Logic}, 35(3):365--374, 09 1970.

\bibitem{GaoBook}
Su~Gao.
\newblock {\em Invariant descriptive set theory}, volume 293 of {\em Pure and
  Applied Mathematics (Boca Raton)}.
\newblock CRC Press, Boca Raton, FL, 2009.

\bibitem{GaoSu}
Su~Gao.
\newblock {\em Invariant Descriptive Set Theory}.
\newblock Taylor and Francis Group, 2009.

\bibitem{sinf}
Noam Greenberg, Alexander Melnikov, Andre Nies, and Daniel Turetsky.
\newblock Effectively closed subgroups of the infinite symmetric group.
\newblock {\em Proc. Amer. Math. Soc.}, 146(12):5421--5435, 2018.

\bibitem{Grubba}
Tanja Grubba and Klaus Weihrauch.
\newblock On computable metrization.
\newblock {\em Electronic Notes in Theoretical Computer Science}, 167:345--364,
  2007.

\bibitem{Uffe}
Uffe Haagerup and Agata Przybyszewska.
\newblock Proper metrics on locally compact groups, and proper affine isometric
  actions on banach spaces.
\newblock 2006.
\newblock Unpublished.

\bibitem{uptohom}
Matthew Harrison-Trainor, Alexander Melnikov, and Keng~Meng Ng.
\newblock Computability of {P}olish spaces up to homeomorphism.
\newblock {\em The Journal of Symbolic Logic}, pages 1--25, 2020.

\bibitem{Hig}
G.~Higman.
\newblock Subgroups of finitely presented groups.
\newblock {\em Proc. Roy. Soc. Ser. A}, 262:455--475, 1961.

\bibitem{topsel}
M.~Hoyrup, T.~Kihara, and V.~Selivanov.
\newblock Degree spectra of homeomorphism types of {P}olish spaces.
\newblock {\em Preprint.}, 2020.

\bibitem{IlKi}
Zvonko Iljazovi\'{c} and Takayuki Kihara.
\newblock Computability of subsets of metric spaces.
\newblock In {\em Handbook of computability and complexity in analysis}, Theory
  Appl. Comput., pages 29--69. Springer, Cham, [2021] \copyright 2021.

\bibitem{Kal1}
Iraj Kalantari and Galen Weitkamp.
\newblock Effective topological spaces. {I}. {A} definability theory.
\newblock {\em Ann. Pure Appl. Logic}, 29(1):1--27, 1985.

\bibitem{Hisa2}
N.~Khisamiev.
\newblock Hierarchies of torsion-free abelian groups.
\newblock {\em Algebra i Logika}, 25(2):205--226, 244, 1986.

\bibitem{Khi}
N.~Khisamiev.
\newblock Constructive abelian groups.
\newblock In {\em Handbook of recursive mathematics, {V}ol.\ 2}, volume 139 of
  {\em Stud. Logic Found. Math.}, pages 1177--1231. North-Holland, Amsterdam,
  1998.

\bibitem{Khas}
N.~Khisamiev and Z.~Khisamiev.
\newblock Nonconstructivizability of the reduced part of a strongly
  constructive torsion-free abelian group.
\newblock {\em Algebra i Logika}, 24:69--76, 1985.

\bibitem{Klee}
V.~L. Klee, Jr.
\newblock Invariant metrics in groups (solution of a problem of {B}anach).
\newblock {\em Proc. Amer. Math. Soc.}, 3:484--487, 1952.

\bibitem{KudKorTop}
Margarita Korovina and Oleg Kudinov.
\newblock The {R}ice-{S}hapiro theorem in computable topology.
\newblock {\em Log. Methods Comput. Sci.}, 13(4):Paper No. 30, 13, 2017.

\bibitem{LaRo1}
Peter La~Roche.
\newblock Effective {G}alois theory.
\newblock {\em J. Symbolic Logic}, 46(2):385--392, 1981.

\bibitem{LaRothesis}
Peter~Edwin La~Roche.
\newblock {\em C{ONTRIBUTIONS} {TO} {RECURSIVE} {ALGEBRA}}.
\newblock ProQuest LLC, Ann Arbor, MI, 1978.
\newblock Thesis (Ph.D.)--Cornell University.

\bibitem{lupini}
M.~Lupini, A.~Melnikov, and A.~Nies.
\newblock Computable topological abelian groups.
\newblock {\em Preprint.}, 2021.

\bibitem{Ma61}
A.~Mal{\cprime}cev.
\newblock Constructive algebras. {I}.
\newblock {\em Uspehi Mat. Nauk}, 16(3 (99)):3--60, 1961.

\bibitem{tdlc}
A.~Melnikov and A.~Nies.
\newblock Computably locally compact totally disconnected groups.
\newblock Preprint., 2022.

\bibitem{Pontr}
Alexander Melnikov.
\newblock Computable topological groups and {P}ontryagin duality.
\newblock {\em Trans. Amer. Math. Soc.}, 370(12):8709--8737, 2018.

\bibitem{MeMo}
Alexander Melnikov and Antonio Montalb\'{a}n.
\newblock Computable {P}olish group actions.
\newblock {\em J. Symb. Log.}, 83(2):443--460, 2018.

\bibitem{separating}
Alexander Melnikov and Keng~Meng Ng.
\newblock Separating notions in computable topology.
\newblock {\em Preprint}, 2022.

\bibitem{MelSurvey}
Alexander~G. Melnikov.
\newblock Computable abelian groups.
\newblock {\em The Bulletin of Symbolic Logic}, 20(3):315--356, 2014.

\bibitem{MetNer79}
G.~Metakides and A.~Nerode.
\newblock Effective content of field theory.
\newblock {\em Ann. Math. Logic}, 17(3):289--320, 1979.

\bibitem{metric2}
Y.~N. Moschovakis.
\newblock Recursive metric spaces.
\newblock {\em Fund. Math.}, 55:215--238, 1964.

\bibitem{Nov:55}
P.~Novikov.
\newblock On the algorithmic unsolvability of the word problem in group theory.
\newblock {\em Trudy Mat. Inst. Steklov}, 44:1--143, 1955.

\bibitem{OdSel-89}
S.~P. Odintsov and V.~L. Selivanov.
\newblock Arithmetic hierarchy and ideals of enumerated {B}oolean algebras.
\newblock {\em Sib. Math. J.}, 30(6):952--960, 1989.

\bibitem{ArnoSurvey}
Arno Pauly.
\newblock On the topological aspects of the theory of represented spaces.
\newblock {\em Computability}, 5(2):159--180, 2016.

\bibitem{Pauly}
Arno Pauly.
\newblock Effective local compactness and the hyperspace of located sets.
\newblock {\em CoRR}, abs/1903.05490, 2019.

\bibitem{ArnoHaar}
Arno Pauly, Dongseong Seon, and Martin Ziegler.
\newblock {Computing Haar Measures}.
\newblock In Maribel Fern{\'a}ndez and Anca Muscholl, editors, {\em 28th EACSL
  Annual Conference on Computer Science Logic (CSL 2020)}, volume 152 of {\em
  Leibniz International Proceedings in Informatics (LIPIcs)}, pages
  34:1--34:17, Dagstuhl, Germany, 2020. Schloss Dagstuhl--Leibniz-Zentrum fuer
  Informatik.

\bibitem{PourElRich}
Marian~B. Pour-El and J.~Ian Richards.
\newblock {\em Computability in analysis and physics}.
\newblock Perspectives in Mathematical Logic. Springer-Verlag, Berlin, 1989.

\bibitem{abR}
M.~Rabin.
\newblock Computable algebra, general theory and theory of computable fields.
\newblock {\em Trans. Amer. Math. Soc.}, 95:341--360, 1960.

\bibitem{rogers}
H.~Rogers.
\newblock {\em Theory of recursive functions and effective computability}.
\newblock MIT Press, Cambridge, MA, second edition, 1987.

\bibitem{SmithThesis}
Rick~L. Smith.
\newblock {\em The theory of profinite groups with effective presentations}.
\newblock ProQuest LLC, Ann Arbor, MI, 1979.
\newblock Thesis (Ph.D.)--The Pennsylvania State University.

\bibitem{Smith1}
Rick~L. Smith.
\newblock Effective aspects of profinite groups.
\newblock {\em J. Symbolic Logic}, 46(4):851--863, 1981.

\bibitem{Spreen3}
Dieter Spreen.
\newblock A characterization of effective topological spaces.
\newblock In {\em Recursion theory week ({O}berwolfach, 1989)}, volume 1432 of
  {\em Lecture Notes in Math.}, pages 363--387. Springer, Berlin, 1990.

\bibitem{Struble}
Raimond Struble.
\newblock Metrics in locally compact groups.
\newblock {\em Compositio Mathematica}, 28(3):217–222, 1974.

\bibitem{Weih_Zheng}
K.~Weihrauch and X.~Zheng.
\newblock Effectiveness of the global modulus of continuity on metric spaces.
\newblock {\em Theoretical Computer Science}, 219(1):439 – 450, 1999.

\bibitem{Wei00}
Klaus Weihrauch.
\newblock {\em Computable analysis}.
\newblock Texts in Theoretical Computer Science. An EATCS Series.
  Springer-Verlag, Berlin, 2000.
\newblock An introduction.

\bibitem{comptop}
Klaus Weihrauch and Tanja Grubba.
\newblock Elementary computable topology.
\newblock {\em J.UCS}, 15(6):1381--1422, 2009.

\bibitem{xu_grubba}
Y.~Xu and T.~Grubba.
\newblock On computably locally compact {Hausdorff} spaces.
\newblock {\em Mathematical Structures in Computer Science}, 19(1):101 -- 117,
  2009.

\end{thebibliography}

\end{document}